\documentclass[letterpaper, 12pt]{article}
\usepackage{amssymb,amsmath}
\usepackage{times,bm,mathrsfs}  
\usepackage{amsmath}
\usepackage{amssymb}
\usepackage{amsfonts}
\usepackage{mathrsfs}
\usepackage{bbm}
\usepackage{color}
\usepackage{graphicx}
\usepackage{amscd}
\usepackage{epsfig}

\newenvironment{proof}{\noindent{\textbf{Proof:}}}{$\blacksquare$\vskip\belowdisplayskip}

\definecolor{Red}{rgb}{1,0,0}
\definecolor{Blue}{rgb}{0,0,1}
\definecolor{Olive}{rgb}{0.41,0.55,0.13}
\definecolor{Green}{rgb}{0,1,0}
\definecolor{MGreen}{rgb}{0,0.8,0}
\definecolor{DGreen}{rgb}{0,0.55,0}
\definecolor{Yellow}{rgb}{1,1,0}
\definecolor{Cyan}{rgb}{0,1,1}
\definecolor{Magenta}{rgb}{1,0,1}
\definecolor{Orange}{rgb}{1,.5,0}
\definecolor{Violet}{rgb}{.5,0,.5}
\definecolor{Purple}{rgb}{.75,0,.25}
\definecolor{Brown}{rgb}{.75,.5,.25}
\definecolor{Grey}{rgb}{.5,.5,.5}
\definecolor{Black}{rgb}{0,0,0}

\newcommand{\dcal}{\mathcal{D}}

\newcommand{\fcal}{\mathcal{F}}

\newcommand{\lcal}{\mathcal{L}}

\newcommand{\qcal}{\mathcal{Q}}
\newcommand{\rcal}{\mathcal{R}}
\newcommand{\scal}{\mathcal{S}}
\newcommand{\tcal}{\mathcal{T}}

\newcommand{\zcal}{\mathcal{Z}}

\newcommand{\real}{\mathbb{R}}

\newcommand{\eps}{\varepsilon}

\newcommand{\ind}{\mathbbm{1}}

\newcommand{\bdm}{\begin{displaymath}}
\newcommand{\edm}{\end{displaymath}}
\newcommand{\bea}{\begin{eqnarray*}}
\newcommand{\eea}{\end{eqnarray*}}
\newcommand{\bean}{\begin{eqnarray}}
\newcommand{\eean}{\end{eqnarray}}

\newcommand{\expec}{\mathbb{E}}
\newcommand{\var}{\mathrm{Var}}
\newcommand{\cov}{\mathrm{Cov}}

\newcommand{\poly}{\mathrm{poly}}

\newtheorem{theorem}{Theorem}
\newtheorem{proposition}{Proposition}

\newtheorem{definition}{Definition}
\newtheorem{example}{Example}
\newtheorem{lemma}{Lemma}

\newtheorem{remark}{Remark}

\newcommand{\ignore}[1]{}

\renewcommand{\poly}{{\mbox{{\rm poly}}}}

 \def\eps{\varepsilon} \def\E{{\mathbb{E}}}

\def\R{\hbox{I\kern-.2em\hbox{R}}}

\def\|{\, | \, }
\def\v0{{\bf 0}}

\def\0{\hat{0}}
\def\1{\hat{1}}

\def\phi{\varphi}

\def\be{\begin{equation}}
\def\ee{\end{equation}}

\def\part{{\cal P}}

\def\R{\mathcal{R}}

\def\eps{\varepsilon}

\newcommand{\s}{X}
\renewcommand{\ss}{Z}
\newcommand{\weight}{\tau}
\newcommand{\eweight}{\hat\tau}
\newcommand{\path}{\mathrm{P}}
\newcommand{\dist}{\mathrm{D}}
\newcommand{\edist}{\ddot{\tau}}
\newcommand{\phy}{\tcal}
\newcommand{\rt}{r}

\newcommand{\states}{[q]}
\newcommand{\nstates}{q}

\newcommand{\evr}{\nu}
\newcommand{\rates}{\mathbb{Q}}
\newcommand{\gtr}{\mathbb{GTR}}
\newcommand{\gmrft}{\mathbb{GMRFT}}
\newcommand{\law}{\dcal}
\newcommand{\sphy}{\mathbb{BY}}

\newcommand{\deep}{\widehat{\mathbb{FP}}}

\newcommand{\longtest}{\widehat{\mathbb{SD}}}

\newcommand{\diag}{\mathrm{diag}}

\newcommand{\bias}{\beta}
\newcommand{\bbias}{\mathcal{B}}
\newcommand{\vvar}{\mathcal{V}}

\newcommand{\nstatespotts}{q}
\newcommand{\qq}{\nstatespotts}

\newcommand{\ebias}{\widehat{\bias}}
\newcommand{\ebbias}{\widehat{\bbias}}

\newcommand{\erho}{\hat\rho}
\newcommand{\triplet}{\widehat{\mathbb{O}}}

\newcommand{\eigenvec}[1]{\nu^{#1}}
\newcommand{\eeigenvec}[1]{\hat\nu^{#1}}
\newcommand{\eigenval}[1]{\lambda_{#1}}

\newcommand{\edgeeigenval}[1]{\theta^{(#1)}}
\newcommand{\eedgeeigenval}[1]{\hat\theta^{(#1)}}
\newcommand{\edgeeigenvec}[1]{\mu^{#1}}
\newcommand{\eedgeeigenvec}[1]{\hat\mu^{#1}}
\newcommand{\eigenratio}{\varrho}

\newcommand{\critks}{g^{\star}_{\mathrm{KS}}}
\newcommand{\critq}{g^{\star}_{Q}}

\newcommand{\weightmax}{\weight^+}

\newcommand{\es}{\widehat{X}}

\newcommand{\neigh}{\mathrm{N}}
\newcommand{\covmat}{\Sigma_{[n]}}
\newcommand{\covmatz}{\Sigma_{[[n]]}}
\newcommand{\ones}{\mathbf{e}}
\renewcommand{\vec}{\mathbf{v}}
\newcommand{\lin}{\lcal}
\newcommand{\desc}[1]{\lfloor#1\rfloor}
\renewcommand{\sb}[1]{\s_{\desc{#1}}}

\newenvironment{app-proof}[1]{\noindent{\textbf{Proof of #1:}}}{$\blacksquare$\vskip\belowdisplayskip}

\author{
Elchanan Mossel\thanks{Weizmann Institute and UC Berkeley.}\and
S\'ebastien Roch\thanks{UCLA. Work supported by NSF grant DMS-1007144.}\and
Allan Sly\thanks{UC Berkeley.}
}

\title{\vspace{0cm}
Robust estimation of latent tree graphical models:
Inferring hidden states with inexact parameters}

\begin{document}

\maketitle

\addtocounter{page}{-1}
\thispagestyle{empty}

\begin{abstract}
Latent tree graphical models are widely
used in computational biology, signal and
image processing, and network tomography.
Here we design a new efficient, estimation
procedure for latent tree models,
including Gaussian and discrete, reversible
models, that significantly improves on
previous sample requirement bounds.
Our techniques are based on a new
hidden state estimator which is robust
to inaccuracies in estimated parameters.
More precisely, we prove that latent tree
models can be estimated with high probability
in the so-called Kesten-Stigum regime
with $O(\log^2 n)$ samples.
\end{abstract}


\clearpage

\section{Introduction}

\paragraph{Background}
Latent tree graphical models
and other related models
have been widely
studied in mathematical statistics, machine learning,
signal and image processing, network tomography, computational
biology, and statistical physics.
See e.g.~\cite{Anderson:58,KollerFriedman:09,Willsky:02,CCLNY:04,SempleSteel:03,
EvKePeSc:00} and references therein.
For instance, in
phylogenetics \cite{Felsenstein:04},
one seeks to reconstruct the evolutionary history
of living organisms
from molecular data extracted from modern species.
The assumption is that molecular data consists of aligned sequences and
that each position in the sequences evolves independently according to a
Markov random field on a tree, where the key parameters are (see Section~\ref{section:definitions} for formal definitions):
\begin{itemize}
\item {\em Tree.}
An evolutionary tree $T$, where the leaves are the modern species and each branching represents a past speciation event.
\item {\em Rate matrix.}
A $\nstates \times \nstates$ mutation rate matrix $Q$, where $\nstates$ is the alphabet size.
A typical alphabet arising in biology would be $\{\mathrm{A},\mathrm{C},\mathrm{G},\mathrm{T}\}$.
Without loss of generality, here we denote the alphabet by $\states = \{1,\ldots,\nstates\}$.
The $(i,j)$'th entry of $Q$ encodes the rate at which state $i$ mutates into state $j$.
We normalize the matrix $Q$ so that its spectral gap is $1$.
\item {\em Edge weights.}
For each edge $e$, we have a scalar branch length $\weight_{e}$ which measures the total amount of evolution along edge $e$.
(We use edge or branch interchangeably.)
Roughly speaking, $\weight_{e}$ is the time elapsed between the end points of $e$. (In fact the time is multiplied by an edge-dependent
overall mutation rate because of our normalization of $Q$.) We also think of $\weight_{e}$ as the ``evolutionary distance''
between the end points of $e$.
\end{itemize}
Other applications, including those involving
Gaussian models
(see Section~\ref{section:definitions}), are similarly
defined. Two statistical
problems naturally arise in this context:
\begin{itemize}
\item
{\em Tree Model Estimation (TME).}
Given $k$ samples of the above process at
the observed nodes, that is, at the leaves of the
tree, estimate the topology of the tree as well as
the edge weights.
\item
{\em Hidden State Inference (HSI).}
Given a fully specified tree model
and a single sample at the observed nodes,
infer the state at the
(unobserved) root of the tree.
\end{itemize}
In recent years, a convergence of techniques
from statistical physics and theoretical
computer science has provided fruitful
new insights on the deep connections between
these two problems, starting with~\cite{Mossel:04a}.

\paragraph{Steel's Conjecture}
A crucial parameter in the second problem above
is $\weightmax(T) = \max_e \weight_{e}$,
the maximal edge weight in the tree.
For instance, for the two-state symmetric $Q$
also known as the Ising model,
it is known that
there exists a critical parameter
$\critks = \ln\sqrt{2}$
such that, if $\weightmax(T) < \critks$,
then it is possible to perform HSI (better than random;
see the Section~\ref{sec:ThmNonrecon} for additional details).
In contrast, if $\weightmax(T) \geq \critks$,
there exist trees for which HSI is impossible, that is, the correlation between the best root estimate and its true value decays exponentially in the depth of the tree.
The regime $\weightmax(T) < \critks$ is known
as the Kesten-Stigum (KS) regime~\cite{KestenStigum:66}.

A striking and insightful conjecture of Steel
postulates a deep connection between TME and HSI~\cite{Steel:01}.
More specifically
the conjecture states that for the Ising model,
in the KS regime,
high-probability TME may be achieved
with a number of samples
$k = O(\log n)$.
Since the number of trees on $n$ labelled leaves is $2^{\Theta(n \log n)}$,
this is an optimal sample requirement
up to constant factors.
The proof of Steel's conjecture was established in~\cite{Mossel:04a} for the Ising model on balanced trees and in~\cite{DaMoRo:11a} for rate matrices on trees with discrete edge lengths.
Furthermore, results of Mossel~\cite{Mossel:03,Mossel:04a}
show that for $\weightmax(T) \geq \critks$ a polynomial sample requirement is needed for correct TME,
a requirement achieved by several estimation algorithms~\cite{ErStSzWa:99a, Mossel:04a, Mossel:07, GrMoSn:08, DaMoRo:11b}.
The previous results have been extended to general
reversible $Q$ on alphabets of size $\nstates\geq 2$~\cite{Roch:10,MoRoSl:11}.
(Note that in that case a more general threshold
$\critq$ may be defined, although little rigorous
work has been dedicated to its study.
See~\cite{Mossel:01,Sly:09,MoRoSl:11}.
In this paper we consider only the KS regime.)

\paragraph{Our contributions}
Prior results for general trees and general rate matrix $Q$,
when $\weightmax(T) < \critks$, have assumed that
edge weights are discretized. This assumption
is required to avoid dealing with the sensitivity
of root-state inference to inexact (that is,
estimated) parameters.
Here we design a new HSI
procedure in the KS regime which is provably
robust to inaccuracies in the parameters (and,
in particular, does not
rely on the discretization assumption).
More precisely, we prove that $O(\log^2 n)$
samples suffice to solve the TME and
HSI problems in the KS regime without
discretization. We consider
two models in detail: discrete, reversible
Markov random fields (also known as GTR models
in evolutionary biology), and Gaussian models.
As far as we know, Gaussian models
have not previously been studied in the context
of the HSI phase transition. (We derive
the critical threshold for Gaussian models
in Section~\ref{sec:ThmNonrecon}.) Formal statements of our results
can be found in Section~\ref{section:results}.
Section~\ref{section:sketch}
provides a sketch of the proof.

\paragraph{Further related work}
For further related work on
sample requirements in tree graphical model estimation,
see \cite{ErStSzWa:99b, MosselRoch:06,
TaAnWi:10, TaAnTo+:11,
ChTaAn+:11,TaAnWi:11,BhRaRo:10}.

\subsection{Definitions}\label{section:definitions}

\noindent\textbf{Trees and metrics.}
Let $T = (V,E)$ be a tree with leaf set $[n]$,
where $[n] = \{1,\ldots,n\}$.
For two leaves $a,b \in [n]$, we denote by $\path(a,b)$
the set of edges on the unique path between $a$ and $b$.
For a node $v \in V$, let $\neigh(v)$ be the neighbors
of $v$.
\begin{definition}[Tree Metric]
A {\em tree metric} on $[n]$ is a positive function
$\dist:[n]\times[n] \to (0,+\infty)$ such that there exists
a tree $T = (V,E)$ with leaf set $[n]$ and an edge weight
function $w:E \to (0,+\infty)$ satisfying the following: for all leaves
$a,b \in [n]$
\begin{equation*}
\dist(a,b) = \sum_{e\in \path(a,b)} w_e.
\end{equation*}
\end{definition}
In this work, we consider dyadic trees.
Our techniques can be extended to complete trees
of higher degree. We discuss general
trees in the concluding remarks.
\begin{definition}[Balanced tree]
A  {\em balanced tree} is a rooted,
edge-weighted, leaf-labeled $h$-level
dyadic tree
$\phy = (V,E,[n],\rt;\weight)$
where:
$h \geq 0$ is an integer;
$V$ is the set of vertices;
$E$ is the set of edges;
$L = [n] = \{1,\ldots,n\}$ is the set of leaves
with $n=2^h$;
$\rt$ is the root;
$\weight : E \to (0,+\infty)$ is a positive edge weight function.
We denote by $\left(\weight(a,b)\right)_{a,b\in [n]}$
the tree metric corresponding to the balanced tree
$\phy = (V,E,[n],\rt;\weight)$.
We extend $\weight(u,v)$ to all vertices $u,v \in V$.
We let $\sphy_n$ be the set of all such balanced
trees on $n$ leaves
and we let $\sphy = \{\sphy_{2^h}\}_{h\geq 0}$.
\end{definition}

\noindent\textbf{Markov random fields on trees.}
We consider Markov models on trees where only
the leaf variables are observed.
The following discrete-state model is standard
in evolutionary biology. See e.g.~\cite{SempleSteel:03}.
Let $\nstates \geq 2$.
Let $\states$ be a state set
and $\pi$ be a distribution on $\states$ satisfying
$\pi_x > 0$ for all $x \in \states$.
The $\nstates\times\nstates$ matrix
$Q$ is a {\em rate matrix} if
$Q_{xy} > 0$ for all $x\neq y$ and
$\sum_{y \in \states} Q_{x y} = 0$,
for all $x \in \states$.
The rate matrix $Q$ is {\em reversible with respect to $\pi$}
if $\pi_x Q_{x y} = \pi_y Q_{y x}$, for all $x, y \in \states$.
By reversibility, $Q$ has $\nstates$
real eigenvalues
$0 = \lambda_1 > \lambda_2 \geq \cdots \geq \lambda_{\nstates}$.
We normalize $Q$ by fixing $\lambda_2 = -1$.
We denote by $\rates_\nstates$ the set of all such rate matrices.
\begin{definition}[General Time-Reversible (GTR) Model]
For $n \geq 1$,
let
$$\phy = (V,E,[n],\rt;\weight)$$
be a balanced tree.
Let $Q$ be a $\nstates\times\nstates$ rate matrix
reversible with respect to $\pi$. Define
the transition matrices $M^e = e^{\weight_e Q}$,
for all $e\in E$.
The GTR model on $\phy$ with rate matrix $Q$
associates a state $\ss_v$ in $\states$ to each
vertex $v$ in $V$ as follows:
pick a state for the root $\rt$ according to $\pi$;
moving away from the root, choose a state for each
vertex $v$ independently according to the distribution
$(M^e_{\ss_u, j})_{j\in\states}$,
with $e = (u,v)$ where $u$ is the parent of $v$.
We let $\gtr_{n,\nstates}$
be the set of all $\nstates$-state GTR models on
$n$ leaves. We denote $\gtr_\nstates =
\left\{\gtr_{2^h,\nstates}\right\}_{h \geq 0}$.
We denote by $\ss_W$ the vector of states on the vertices
$W\subseteq V$. In particular, $\ss_{[n]}$ are the states
at the leaves.
We denote by $\law_{\phy,Q}$ the distribution of $\ss_{[n]}$.
\end{definition}
GTR models encompass several special cases such as the Cavender-Farris-Neyman (CFN) model
and the Jukes-Cantor (JC) model.
\begin{example}[$\qq$-state Symmetric Model]\label{ex:symmetric}
The {\em $\qq$-state Symmetric model} (also call\-ed $\qq$-state Potts model) is the GTR model with $\qq \geq 2$ states,
$\pi = (1/\qq,\ldots, 1/\qq)$,
and $Q = Q^{\qq\mathrm{-POTTS}}$ where
\begin{equation*}
Q^{\qq\mathrm{-POTTS}}_{ij}
=
\left\{
\begin{array}{ll}
-\frac{\qq - 1}{\qq} & \mbox{if $i=j$}\\
\frac{1}{\qq} & \mbox{o.w.}
\end{array}
\right.
\end{equation*}
Note that $\lambda_2(Q) = -1$.
The special cases $\qq=2$ and $\qq=4$ are called respectively the CFN and JC models in the biology literature.
We denote their rate matrices by $Q^{\mathrm{CFN}}, Q^{\mathrm{JC}}$.
\end{example}
A natural generalization of the CFN model which is also included in the GTR framework is the Binary Asymmetric Channel.
\begin{example}[Binary Asymmetric Channel]
Letting $\nstates = 2$ and
$\pi = (\pi_{1}, \pi_{2})$,
with $\pi_{1},\pi_{2} > 0$, we can take
\begin{equation*}
Q
=
\left(
\begin{array}{cc}
-\pi_{2} & \pi_{2}\\
\pi_{1} & -\pi_{1}
\end{array}
\right).
\end{equation*}
\end{example}
The following transformation
will be useful~\cite{MosselPeres:03}.
Let $\evr$ be a right eigenvector of the GTR matrix $Q$ corresponding to the eigenvalue $-1$.
Map the state space to the real line by defining $\s_x = \evr_{\ss_x}$ for all $x\in[n]$.

We also consider Gaussian Markov
Random Fields on Trees (GMRFT).
Gaussian graphical models,
including Gaussian tree models,
are common in statistics, machine
learning as well as signal
and image processing.
See e.g.~\cite{Anderson:58,Willsky:02}.
\begin{definition}[Gaussian Markov Random Field
on a Tree (GMRFT)]
For $n \geq 1$,
let $\phy = (V,E,[n],\rt;\weight)$ be a balanced tree.
A GMRFT on $\phy$ is a multivariate Gaussian vector
$\s_V = (\s_{v})_{v \in V}$ whose covariance matrix $\Sigma
= (\Sigma_{u v})_{u,v\in V}$ with inverse
$\Lambda = \Sigma^{-1}$ satisfies
$$
(u,v) \notin E, u \neq v \implies \Lambda_{uv} = 0.
$$
We assume that only the states at the
leaves $\s_{[n]}$ are observed.
To ensure identifiability (that is,
to ensure that two different sets of
parameters generate different distributions
at the leaves), we assume that all
internal nodes have zero mean and
unit variance and that all non-leaf edges
correspond to a nonnegative correlation.
Indeed shifting and scaling the states at the
internal nodes does not affect the leaf
distribution.
For convenience, we extend this assumption
to leaves and leaf edges.
With the choice
$$
\Sigma_{uv}
= \prod_{e \in \path(u,v)} \rho_e,
\quad u,v \in V,
$$
where $\rho_e = e^{-\weight_e}$, for all $e \in E$,
a direct calculation shows that
\begin{displaymath}
\Lambda_{u v}
= \begin{cases}
1 + \sum_{w \in \neigh(v)} \frac{\rho_{(v, w)}^2}
{1 - \rho_{(v, w)}^2}, & \text{if $u = v$},\\
- \frac{\rho_{(u, v)}}{1 - \rho_{(u, v)}^2},
& \text{if $(u,v) \in E$},\\
0, & \text{o.w.}
\end{cases}
\end{displaymath}
(Note that, in computing $(\Sigma\Lambda)_{u v}$
with $u\neq v$,
the product $\prod_{e \in \path(u,w)} \rho_e$ factors out, where
$w \in \neigh(v)$ with $(w,v) \in \path(u,v)$.)
In particular,
$\{-\log|\Sigma_{uv}|\}_{uv \in [n]}$ is a tree
metric.
We denote by $\law_{\phy,\Sigma}$ the distribution
of $\s_{[n]}$.
We let $\gmrft_{n}$
be the set of all GMRFT models on
$n$ leaves. We denote $\gmrft =
\left\{\gmrft_{2^h}\right\}_{h \geq 0}$.
\end{definition}
\begin{remark}
Our techniques extend to cases
where leaves and leaf edges
have general means and covariances.
We leave the details to the reader.
\end{remark}
Equivalently, in a formulation closer
to that of the GTR model above,
one can think of a GMRFT model
as picking a root value according
to a standard Gaussian distribution
and running independent
Ornstein-Uhlenbeck processes on the edges.

Both the GTR and GMRFT models are
{\em globally Markov}: for all disjoint
subsets $A,B,C$ of $V$ such
that $B$ {\em separates} $A$ and $C$,
that is, all paths between $A$ and $C$
go through a node in $B$, we have
that the states at $A$ are
conditionally independent
of the states at $C$ given the states at $B$.

\subsection{Results}
\label{section:results}

Our main results are the following. We are given
$k$ i.i.d.~samples from a GMRFT or GTR model
and we seek to estimate the tree structure
with failure probability going to $0$ as the
number of leaves $n$ goes to infinity.
We also estimate edge weights within
constant tolerance.
\begin{theorem}[Main Result: GMRFT Models]
\label{thm:maingaussian}
Let $0 < f < g < +\infty$
and denote by $\gmrft^{f,g}$ the set of
all GMRFT models
on balanced trees $\phy = (V,E,[n],\rt;\weight)$ satisfying
$f < \weight_e < g,\ \forall e\in E$.
Then, for all $0 < f < g < \critks = \ln \sqrt{2}$,
the tree structure estimation problem on
$\gmrft^{f,g}$
can be solved with $k = \kappa \log^2 n$
samples,
where $\kappa = \kappa(f,g) > 0$
is large enough. Moreover all edge weights
are estimated within constant tolerance.
\end{theorem}
This result is sharp as we prove the following negative results establishing the equivalence of the TME and HSI thresholds.
\begin{theorem}\label{thm:gaussianNonRecon}
If $0 < f \leq g$ with $g > \critks = \ln \sqrt{2}$,
then the tree structure estimation problem on
$\gmrft^{f,g}$
cannot, in general, be solved without at least $k = n^\gamma$
samples, where $\gamma = \gamma(f,g) > 0$.
\end{theorem}
The proof of the theorem is in Section~\ref{section:gaussian}.
\begin{theorem}[Main Result: GTR Models]
\label{thm:maingtr}
Let $0 < f < g < +\infty$
and denote by $\gtr_\nstates^{f,g}$ the set of
all $\nstates$-state GTR models
on balanced trees $\phy = (V,E,[n],\rt;\weight)$ satisfying
$f < \weight_e < g,\ \forall e\in E$.
Then, for all $\nstates \geq 2$,
$0 < f < g < \critks = \ln \sqrt{2}$,
the tree structure estimation problem on
$\gtr_\nstates^{f,g}$
can be solved with $k = \kappa \log^2 n$
samples,
where $\kappa = \kappa(\nstates,f,g) > 0$
is
large enough. Moreover all edge weights
are estimated within constant tolerance.
\end{theorem}
The proof of this theorem is similar to that
of Theorem~\ref{thm:maingaussian}.
However dealing with unknown rate matrices
requires some care
and the full proof of the modified algorithm in that case
can be found in Section~\ref{section:gtr}.
\begin{remark}
Our techniques extend to $d$-ary trees
for general (constant) $d \geq 2$. In that
case, the critical threshold satisfies
$d e^{-2\weight} = 1$. We leave the details
to the reader.
\end{remark}

\subsection{Proof Overview}\label{section:sketch}

We give a sketch of the proof of our main result.
We discuss the case of GTR models with known
$Q$ matrix. The unknown $Q$ matrix and
Gaussian cases are similar. See Sections~\ref{section:gaussian} and \ref{section:gtr} for details.
Let $(Z_{[n]}^i)_{i=1}^k$ be i.i.d.~samples
from a GTR model on a balanced tree with
$n$ leaves. Let $(Z_V)$ be a generic sample
from the GTR model.

\paragraph{Boosted algorithm}
As a starting point, our algorithm uses the reconstruction framework of \cite{Mossel:04a}.
This basic ``boosting'' approach is twofold:
\begin{itemize}
\item {\em Initial Step.} Build the first level of the tree from the samples at the leaves.
This can be done easily
by standard quartet-based techniques. (See Section~\ref{section:topology}.)
\item {\em Main Loop.} Repeat the following two steps until the tree is built:
\begin{enumerate}
\item {\em HSI.} Infer hidden states at the roots of the reconstructed subtrees.
\item {\em One-level TME.} Use the hidden state estimates from the previous step to build the next level of the tree
using quartet-based techniques.
\end{enumerate}
\end{itemize}
The heart of the procedure is Step 1. Note that, assuming each level is correctly
reconstructed, the HSI problem in Step 1 is performed on a known, correct topology.
However the edge weights are unknown and
need to be estimated from the samples at the leaves.

This leads to the key technical issue addressed
in this paper. Although HSI with known topology
and edge weights is well understood (at least
in the so-called Kesten-Stigum (KS) regime~\cite{MosselPeres:03}), little work has considered the effect of inexact
parameters on hidden state estimation, with the
notable exception
of~\cite{Mossel:04a} where a parameter-free estimator
is developed for the Ising model.
The issue was averted in prior work on GTR models
by assuming that edge weights are discretized,
allowing exact estimation~\cite{DaMoRo:11a,Roch:10}.

Quartet-based tree structure and edge weight
estimation relies on the following distance estimator.
It is natural to use a distance estimator involving
the eigenvectors of $Q$.
Let $\evr$ be a second right eigenvector of the
GTR matrix $Q$ corresponding to the eigenvalue $-1$.
For $a \in V$ and $i=1,\ldots,k$,
map the samples to the real line
by defining
$\s^i_{a} = \evr_{\ss^i_a}$.
Then define
\begin{equation}\label{eq:eigenestim2}
\hat\weight(a,b) = -\ln \left(\frac{1}{k}\sum_{i=1}^k \s^i_a \s^i_b\right).
\end{equation}
It can be shown that:
For all $a,b\in V$, we have
$-\ln\expec[e^{-\hat\weight(a,b)}] = \weight(a,b)$.
Note that, in our case, this estimate is only available for pairs of {\em leaves}. Moreover,
it is known that the quality of this estimate degrades quickly as $\weight(a,b)$
increases~\cite{ErStSzWa:99a,Atteson:99}.
To obtain accuracy $\eps$ on a $\weight$ distance
with inverse polynomial failure probability
requires
\begin{equation}\label{eq:slr}
k \geq C_1 \eps^{-2} e^{C_2 \weight} \log n
\end{equation}
samples,
where $C_1, C_2$ are constants.
We use HSI to replace the $X$'s in
(\ref{eq:eigenestim2}) with approximations of
hidden states in order to improve the accuracy
of the distance estimator between {\em internal}
nodes.

\paragraph{Weighted majority}
For the symmetric CFN model with state space $\{+1,-1\}$,
hidden states can be inferred using a linear combination of the states at the leaves---a type of weighted majority vote.
A natural generalization of this linear estimator in the context of more general mutation matrices was studied by~\cite{MosselPeres:03}.
The estimator at the root $\rt$ considered in~\cite{MosselPeres:03} is of the form
\begin{equation}\label{eq:flow}
S_\rt = \sum_{x \in [n]} \left(\frac{\Psi(x)}{e^{-\weight(\rt,x)}}\right) \s_x,
\end{equation}
where $\Psi$ is a unit flow between $\rt$ and $[n]$.
For any such $\Psi$, $S_\rt$ is a conditionally unbiased estimator of $\s_\rt$, that is,
$\expec[S_\rt\,|\,\s_\rt] = \s_\rt$.
Moreover, in the KS regime, that is, when $\weightmax < \critks$,
one can choose a flow such that the variance of $S_\rt$ is uniformly bounded~\cite{MosselPeres:03} and, in fact,
we have the following stronger moment condition
\begin{equation*}
\expec[\exp(\zeta S_\rt)|\s_\rt]
\leq \exp(\zeta \s_\rt + c\zeta^2)
\end{equation*}
for all $\zeta \in \mathbb{R}$~\cite{PeresRoch:11}.
In~\cite{Roch:10} this estimator was used in Step 1 of the boosted algorithm.
On a balanced tree with $\log n$ levels,
obtaining sufficiently accurate estimates
of the coefficients in (\ref{eq:flow}) requires accuracy
$1/\Omega(\log (n))$ on the edge weights.
By (\ref{eq:slr}),
such accuracy requires a $O(\log^3 n)$ sequence
length.
Using misspecified edge
weights in (\ref{eq:flow}) may lead to a
highly biased estimate and generally may fail to give a good reconstruction at the root.
Here we achieve accurate hidden state estimation
using only $O(\log^2 n)$ samples.

\paragraph{Recursive estimator}
We propose to construct an estimator of the form (\ref{eq:flow}) \emph{recursively}.
For $x\in V$ with children $y_1, y_2$ we let
\begin{equation}\label{eq:recursive}
S_v = \omega_{y_1} S_{y_1} + \omega_{y_2} S_{y_2},
\end{equation}
and choose the coefficients $\omega_{y_1}, \omega_{y_2}$
to guarantee the following conditions:
\begin{itemize}
\item We have
$$
\expec[S_x \,|\,\ss_x] = \bbias(x)\s_x,
$$
with a bias term $\bbias(x)$ close to 1.
\item  The estimator satisfies the exponential moment condition
$$
\expec[\exp(\zeta S_x)|\ss_x] \leq \exp(\zeta \s_x + c\zeta^2).
$$
\end{itemize}
We show that these conditions can be guaranteed provided the model is in the KS regime.
To do so, the procedure measures the bias terms $\bbias(y_1)$ and $\bbias(y_2)$ using methods similar to distance estimation.
By testing the bias and, if necessary, compensating for any previously introduced error, we can adaptively choose
coefficients $\omega_1,\omega_2$ so that $S_x$ satisfies these two conditions.

\paragraph{Unknown rate matrix}
Further complications arise when the matrix $Q$ is not given and has to be estimated from the data.
We give a procedure for recovering $Q$ and
an estimate of its second right eigenvector.
Problematically, any estimate $\hat\evr$ of $\evr$ may have a small
component in the direction of the first right eigenvector of $Q$.
Since the latter has eigenvalue $0$, its component builds up
over many recursions and it eventually
overwhelms the signal.  However, we make use of the fact that the first right eigenvector is identically 1:
by subtracting from $S_x$ its empirical mean, we show that we can cancel the effect of the first eigenvector.
With a careful analysis, this improved procedure leads to an accurate estimator.

\section{Gaussian Model}
\label{section:gaussian}

In this section, we prove our main theorem
in the Gaussian case.
The proof is based on a
new hidden state estimator
which is described in Section~\ref{section:recursive}.
For $n = 2^h$ with$h \geq 0$,
let $\phy = (V,E,[n],\rt;\weight)$ be a balanced tree.
We assume that
$0 \leq \weight_e < g,\ \forall e\in E$,
with $ 0 < g < \critks = \ln \sqrt{2}$.
The significance of the threshold
$\critks$ is explained in Section~\ref{sec:ThmNonrecon} where we also prove Theorem~\ref{thm:gaussianNonRecon}.
We generate $k$ i.i.d.~samples
$(\s^i_{[n]})_{i=1}^k$ from the GMRFT model
$\law_{\phy, \Sigma}$ where $k = \kappa \log^2 n$.

Our construction is recursive, building the tree
and estimating hidden states one level at a time.
To avoid unwanted correlations, we use a fresh
block of samples for each level.
Let $K = \kappa \log n$ be the size of each block.

\subsection{Recursive Linear Estimator}
\label{section:recursive}

The main tool in our reconstruction algorithm
is a new hidden state estimator. This estimator
is recursive, that is, for a node
$x \in V$ it is constructed from estimators
for its children $y, z$. In this subsection,
we let $\s_V$ be a generic sample from
the GMRFT independent of
everything else. We let $(\s_{[n]}^i)_{i=1}^K$
be a block of independent samples at the leaves.
For a node $u \in V$, we let $\desc{u}$ be the
leaves below $u$ and $\sb{u}$, the corresponding
state.

\paragraph{Linear estimator}
We build a {\em linear} estimator for each of the vertices recursively from the leaves.
Let $x\in V-[n]$ with children (direct descendants)
$y_1,y_2$.
Assume that the topology of the tree rooted at $x$
has been correctly reconstructed, as detailed
in Section~\ref{section:topology}.
Assume further that we have constructed linear
estimators
$$S_u \equiv \lin_u(\sb{u})$$
of $\s_u$, for all $u \in V$ below $x$.
We use the convention that $\lin_u(\sb{u}) = \s_u$
if $u$ is a leaf.
We let $\lin_x$ be a linear
combination of the form
\begin{equation}\label{eq:linear}
S_x \equiv \lin_x(\sb{x}) =
\omega_{y_1} \lin_{y_1}(\sb{y_1}) +
\omega_{y_2} \lin_{y_2}(\sb{y_2}),
\end{equation}
where---ideally---the $\omega$'s are chosen so as to satisfy the following conditions:
\begin{enumerate}
\item {\bf Unbiasedness.} The estimator
$S_x = \lin_x(\sb{x})$ is {\em conditionally unbiased}, that is,
$$
\expec[S_x\,|\,\s_x] = \s_x.
$$

\item {\bf Minimum Variance.} The estimator
has minimum variance amongst
all estimators of the form (\ref{eq:linear}).
\end{enumerate}
An estimator with these properties
can be constructed given exact knowledge of the
edge parameters, see Section~\ref{sec:ThmNonrecon}.
However, since the edge parameters
can only be estimated with constant accuracy
given the samples, we need a procedure that
satisfies these conditions only approximately.
We achieve this
by 1) recursively minimizing the variance at each level
and 2) at the same time measuring the bias and adjusting for any deviation that may have
accumulated from previously estimated branch lengths.

\paragraph{Setup}
We describe the basic recursive step of our
construction.
As above, let $x\in V-[n]$ with children $y_1,y_2$ and corresponding edges $e_1 = (x,y_1), e_2=(x,y_2)$.
Let $0 < \delta < 1$ (small) and $c > 1$ (big) be constants to be defined later.
Assume that we have the following:
\begin{itemize}
\item Estimated edge weights
$\eweight_{e}$ for all edges $e$ below $x$
such that there is $\eps > 0$
with
\begin{equation}\label{eq:weightassumption}
|\eweight_{e} - \weight_{e}| < \eps.
\end{equation}
The choice of $\eps$ and the procedure to obtain these estimates are described in Section~\ref{section:weights}.
We let $\hat\rho_{e} = e^{-\eweight_{e}}$.

\item Linear estimators $\lin_{u}$
for all $u\in V$ below $x$ such that
with
\begin{equation}\label{eq:biasguarantee}
\expec[S_{u} \,|\, \s_{u}] = \bbias(u) \s_{u},
\end{equation}
where $S_u \equiv \lin_u(\sb{u})$,
for some $\bbias(u) > 0$ with $|\bbias(u) - 1| < \delta$
and
\begin{equation}\label{eq:varguarantee}
\vvar(u)
\equiv \var[S_u]
\leq c.
\end{equation}
Note that these conditions are satisfied at the leaves.
Indeed, for $u \in [n]$ one has
$S_{u} = \s_{u}$
and therefore
$\expec[S_{u}\,|\,\s_{u}] = \s_{u}$
and
$\vvar(u) = \var[X_u] = 1$.
We denote $\beta(u) = -\ln \bbias(u)$.

\end{itemize}
We now seek to construct $S_x$ so that it in turn satisfies the same conditions.
\begin{remark}
In this subsection, we are treating the estimated
edge weights and linear estimator coefficients
as deterministic. In fact, they are random variables
depending on sample blocks used on prior recurrence
levels---and in particular they are independent of
$X_V$ and of
the block of samples used on the current level.
\end{remark}

\paragraph{Procedure}
Given the previous setup, we choose the weights $\omega_{y_\alpha}$, $\alpha=1,2$, as follows.
For $u,v \in V$ below $x$ and
$\ell =1,\ldots,K$ let
$$
S_u^\ell \equiv \lin_u(\sb{u}^\ell),
$$
and define
\begin{equation*}
\edist(u,v) = -\ln\left(\frac{1}{K} \sum_{\ell=1}^K S^\ell_{u}S^\ell_{v}\right),
\end{equation*}
the estimated path length between $u$ and $v$ including bias.
We let $\bias(u) = -\ln \bbias(u)$.
\begin{enumerate}
\item {\bf Estimating the Biases.}
If $y_1, y_2$ are leaves, we let $\ebias(y_\alpha) = 0$, $\alpha=1,2$.
Otherwise, let $z_{21},z_{22}$ be the children of $y_2$. We then compute
\begin{equation*}
\ebias(y_1) = \frac{1}{2}(\edist(y_1, z_{21}) + \edist(y_1, z_{22}) - \edist(z_{21}, z_{22}) - 2\eweight_{e_1} - 2\eweight{e_2}),
\end{equation*}
and similarly for $y_2$.
Let $\ebbias(y_\alpha) = e^{-\ebias(y_\alpha)}$, $\alpha=1,2$.

\item {\bf Minimizing the Variance.}
For $\alpha=1,2$ we set $\omega_{y_1}, \omega_{y_2}$ as
\begin{equation}\label{eq:solution}
\omega_{y_\alpha} = \frac{\ebbias(y_\alpha)\erho_{e_\alpha}}{\ebbias(y_1)^2\erho_{e_1}^2 + \ebbias(y_2)^2\erho_{e_2}^2},
\end{equation}
which corresponds to the solution of the following optimization problem:
\begin{equation}\label{eq:optimization}
\min\{\omega_{y_1}^2 + \omega_{y_2}^2 \ :\ \omega_{y_1} \ebbias(y_1) \erho_{e_1} +  \omega_{y_2} \ebbias(y_2) \erho_{e_2} = 1,\ \omega_{y_1},\omega_{y_2} > 0 \}.
\end{equation}
The constraint in the optimization above is meant to
ensure that the bias condition (\ref{eq:biasguarantee}) is satisfied. We set
\begin{equation*}
\lin_x(\sb{x}) = \omega_{y_1} \lin_{y_1}(\sb{y_1})
+ \omega_{y_2} \lin_{y_2}(\sb{y_2}).
\end{equation*}

\end{enumerate}

\paragraph{Bias and Variance}
We now prove (\ref{eq:biasguarantee}) and (\ref{eq:varguarantee}) recursively assuming
(\ref{eq:weightassumption}) is satisfied.
This follows from the following propositions.
\begin{proposition}[Concentration of Internal Distance Estimates]\label{proposition:concentration}
For all $\eps > 0$, $\gamma > 0$, $0 < \delta < 1$ and
$c > 0$, there is $\kappa = \kappa(\eps,\gamma,\delta,c) > 0$ such that,
with probability at least $1 - O(n^{-\gamma})$, we have
\begin{equation*}
|\edist(u,v) - (\weight(u,v) + \bias(u) + \bias(v))| < \eps,
\end{equation*}
for all $u,v \in \{y_1,y_2,z_{11},z_{12},z_{21},z_{22}\}$ where $z_{\alpha 1}, z_{\alpha 2}$ are the children
of $y_\alpha$.
\end{proposition}
\begin{proof}
First note that
\begin{eqnarray*}
\E\left[
\frac{1}{K} \sum_{\ell=1}^K S^\ell_{u}S^\ell_{v}
\right]
&=& \E\left[
S_{u}S_{v}
\right]\\
&=& \E\left[\E\left[
S_{u}S_{v}|\s_u,\s_v
\right]\right]\\
&=& \E\left[\E\left[
S_{u}|\s_u
\right]\E\left[
S_{v}|\s_v
\right]\right]\\
&=& \E\left[\bbias(u)\bbias(v)\s_u\s_v\right]\\
&=& \bbias(u)\bbias(v)\Sigma_{uv},
\end{eqnarray*}
where we used the Markov property on the third line,
so that
$$
-\ln\left(\E\left[
\frac{1}{K} \sum_{\ell=1}^K S^\ell_{u}S^\ell_{v}
\right]\right)
= \weight(u,v) + \bias(u) + \bias(v).
$$
Moreover, by assumption, $S_u$ is Gaussian
with
$$
\E[S_u] = 0, \quad \var[S_u] = \vvar(u) \leq c,
$$
and similarly for $u$.
It is well-known that in the Gaussian case
empirical covariance estimates as above
have $\chi^2$-type distributions~\cite{Anderson:58}.
Explicitly, note that from
$$
S_u S_v = \frac{1}{2}[(S_u + S_v)^2
- S_u^2 - S_v^2],
$$
it suffices to consider the concentration of
$S_u^2$, $S_v^2$, and $(S_u+S_v)^2$.
Note that
$$
\var[S_u + S_v] = \vvar(u) + \vvar(v) + 2\bbias(u)\bbias(v)\Sigma_{uv} \leq 2c + 2(1+\delta)^2 < +\infty,
$$
independently of $n$.
We argue about $S_u^2$, the other
terms being similar.
By definition, $S_u^2/\vvar(u)$ has a $\chi^2_1$
distribution so that
\begin{equation}\label{e:GaussianSecondMoment}
\E\left[e^{\zeta S_u^2}\right]
= \frac{1}{\sqrt{1 - 2\zeta\vvar(u)}} < +\infty,
\end{equation}
for $|\zeta|$ small enough,
independently of $n$. The proposition
then follows from standard large-deviation
bounds~\cite{Durrett:96}.
\end{proof}

\begin{proposition}[Recursive Linear Estimator: Bias]\label{proposition:bias}
For all $\delta > 0$,
there is $\eps > 0$ small enough so that,
assuming that Proposition~\ref{proposition:concentration}
holds,
$$
\expec[S_{x} \,|\, \s_{x}] = \bbias(x) \s_{x},
$$
for some $\bbias(x) > 0$ with $|\bbias(x) - 1| < \delta$.
\end{proposition}
\begin{proof}
We first show that the conditional biases at $y_1, y_2$ are accurately estimated. From Proposition~\ref{proposition:concentration},
we have
\begin{equation*}
|\edist(z_{2 1},z_{2 2})
- (\weight(z_{2 1},z_{2 2})
+ \bias(z_{2 1}) + \bias(z_{2 2}))|
< \eps,
\end{equation*}
and similarly for $\edist(y_1,z_{21})$ and $\edist(y_1,z_{22})$. Then from (\ref{eq:weightassumption}), we get
\begin{eqnarray*}
2\ebias(y_1)
&=&
\edist(y_1, z_{21}) + \edist(y_1, z_{22})
- \edist(z_{21}, z_{22})
- 2\eweight_{e_1} - 2\eweight_{e_2}\\
&\leq& (\weight(y_1, z_{21}) + \bias(y_1) + \bias(z_{21}))
+ (\weight(y_1, z_{22}) + \bias(y_1) + \bias(z_{22}))\\
&& \qquad - (\weight(z_{21}, z_{22}) + \bias(z_{21}) + \bias(z_{22}))
- 2\weight_{e_1} - 2\weight_{e_2} + 7\eps\\
&=& 2\bias(y_1) + (\weight(y_1, z_{21}) + \weight(y_1, z_{22})
- \weight(z_{21}, z_{22})) - 2(\weight_{e_1} + \weight_{e_2}) + 7\eps\\
&=& 2\bias(y_1) + ([\weight(y_1,y_2) + \weight(y_2,z_{21})]
+ [\weight(y_1,y_2) + \weight(y_2,z_{22})]\\
&&\qquad  - [\weight(z_{21},y_2) + \weight(y_2,z_{22})])
- 2\weight(y_1,y_2) + 7\eps\\
&=& 2\bias(y_1) + 7\eps,
\end{eqnarray*}
where we used the additivity of $\weight$ on line 4. And similarly for the other direction so that
\begin{equation*}
|\ebias(y_1) - \bias(y_1)| \leq \frac{7}{2}\eps.
\end{equation*}
The same inequality holds for $y_2$.

Given $\omega_{y_1}, \omega_{y_2}$, the bias at $x$ is
\begin{eqnarray*}
\expec[S_{x} \,|\, \s_{x}]
&=& \expec[\omega_{y_1}S_{y_1} + \omega_{y_2}S_{y_2} \,|\, \s_{x}]\\
&=& \sum_{\alpha=1,2} \omega_{y_\alpha}
\E[\expec[S_{y_\alpha} \,|\, \s_{y_\alpha},\s_x]|\s_x]\\
&=& \sum_{\alpha=1,2} \omega_{y_\alpha}
\E[\expec[S_{y_\alpha} \,|\, \s_{y_\alpha}]|\s_x]\\
&=& \sum_{\alpha=1,2} \omega_{y_\alpha}
\E[\bbias(y_\alpha)\s_{y_\alpha}|\s_x]\\
&=& (\omega_{y_1} \bbias(y_1) \rho_{e_1} +  \omega_{y_2} \bbias(y_2) \rho_{e_2})\s_x\\
&\equiv& \bbias(x) \s_x,
\end{eqnarray*}
where we used the Markov property on line 2 and the fact that $\s_V$ is Gaussian on line 5. The last line
is a definition.
Note that by the inequality above we have
\begin{eqnarray*}
\bbias(x)
&=& \omega_{y_1} \bbias(y_1) \rho_{e_1} +  \omega_{y_2} \bbias(y_2) \rho_{e_2}\\
&=& \omega_{y_1} e^{-\bias(y_1)} \rho_{e_1} +  \omega_{y_2} e^{-\bias(y_2)} \rho_{e_2}\\
&\leq& \omega_{y_1} e^{-\ebias(y_1) +7/2\eps} (\erho_{e_1} + \eps) +  \omega_{y_2} e^{-\ebias(y_2)+7/2\eps} (\erho_{e_2} + \eps)\\
&=& (\omega_{y_1} \ebbias(y_1) \erho_{e_1} +  \omega_{y_2} \ebbias(y_2) \erho_{e_2}) + \max\{\omega_{y_1}, \omega_{y_2}\}O(\eps)\\
&=& 1 + \max\{\omega_{y_1}, \omega_{y_2}\}O(\eps),
\end{eqnarray*}
where the last line follows from the definition
of $\omega_{y_\alpha}$.
Taking $\eps, \delta$ small enough,
from our previous bounds and equation \eqref{eq:solution},
we can derive that $\omega_{y_\alpha} = O(1)$, $\alpha=1,2$. In particular,
$\bbias(x) = 1+O(\eps)$ and, choosing $\eps$ small enough, it satisfies $|\bbias(x) - 1| < \delta$.
\end{proof}

\begin{proposition}[Recursive Linear Estimator: Variance]\label{proposition:variance}
There exists $c > 0$ large enough
and $\eps, \delta > 0$ small enough
such that,
assuming that Proposition~\ref{proposition:concentration}
holds, we have
$$
\vvar(x) \equiv \var[S_x] \leq c.
$$
\end{proposition}
\begin{proof}
From (\ref{eq:solution}),
\begin{eqnarray*}
\omega_{y_1}^2 + \omega_{y_2}^2
&=& \left(\frac{\rho_{e_1}^2}{(\rho_{e_1}^2 + \rho_{e_2}^2)^2} + \frac{\rho_{e_2}^2}{(\rho_{e_1}^2 + \rho_{e_2}^2)^2}\right)(1 + O(\eps + \delta))\\
&=& \left(\frac{1}{\rho_{e_1}^2 + \rho_{e_2}^2}\right)(1 + O(\eps + \delta))\\
&\leq& \frac{1}{2(\rho^*)^2}(1 + O(\eps + \delta)) < 1,
\end{eqnarray*}
for $\eps, \delta > 0$ small enough, where
$\rho^* = e^{-g}$ so that $2(\rho^*)^2 > 1$.
Moreover,
\begin{eqnarray*}
\var[S_x]
&=& \var[\omega_{y_1} S_{y_1}
+ \omega_{y_2} S_{y_2}]\\
&=& \omega_{y_1}^2 \var[S_{y_1}]
+ \omega_{y_2}^2 \var[S_{y_2}]
+ \omega_{y_1}\omega_{y_2} \E[S_{y_1}S_{y_2}]\\
&\leq& (\omega_{y_1}^2
+ \omega_{y_2}^2) c
+ \omega_{y_1}\omega_{y_2}
\bbias(y_1) \bbias(y_2) \Sigma_{uv}\\
&\leq& (\omega_{y_1}^2
+ \omega_{y_2}^2) c
+ \omega_{y_1}\omega_{y_2}
(1+\delta)^2\\
&<& c,
\end{eqnarray*}
taking $c$ large enough.
\end{proof}

\subsection{Topology reconstruction}
\label{section:topology}

Propositions~\ref{proposition:bias} and~\ref{proposition:variance}
rely on the knowing the topology below $x$.
In this section, we show how this is performed inductively.
That is, we assume the topology is known up to
level  $0 \leq h' < h$ and that hidden state estimators
have been derived up to that level. We then construct
the next level of the tree.

\paragraph{Quartet Reconstruction}
Let $L_{h'}$ be the set of
vertices in $V$ at level $h'$ from the leaves
and let
$\qcal = \{a,b,c,d\} \subseteq L_{h'}$
be a $4$-tuple on level $h'$.
The topology of $T$ restricted to $\qcal$
is completely characterized by a bipartition or
{\em quartet split} $q$ of the form:
$a b | c d$, $a c | b d$ or
$a d | b c$.
The most basic operation in quartet-based reconstruction
algorithms is the inference of such quartet splits.
This is done by performing a {\em four-point test}:
letting
\begin{equation*}
\fcal(a b | c d)
= \frac{1}{2}[\weight(a,c) + \weight(b,d)
- \weight(a,b) - \weight(c,d)],
\end{equation*}
we have
\begin{equation*}
q
=
\left\{
\begin{array}{ll}
a b | c d & \mathrm{if\ }\fcal(a,b|c,d) > 0\\
a c | b d & \mathrm{if\ }\fcal(a,b|c,d) < 0\\
a d | b c & \mathrm{o.w.}
\end{array}
\right.
\end{equation*}
Note however that we cannot estimate directly
the values
$\weight(a,c)$, $\weight(b,d)$,
$\weight(a,b)$, and $\weight(c,d)$
for internal nodes, that is, when $h' > 0$.
Instead we use the internal estimates
described in Proposition~\ref{proposition:concentration}.

\paragraph{Deep Four-Point Test} Let $D > 0$.
We let
\begin{equation*}
\widehat\fcal(a b | c d)
= \frac{1}{2}[\edist(a,c)
+ \edist(b,d)
- \edist(a,b)
- \edist(c,d)],\label{eq:fcal}
\end{equation*}
and
\begin{equation*}
\longtest(\scal) = \ind\{\edist(x,y) \leq D,\ \forall x,y\in\scal\}.
\end{equation*}
We define the {\em deep four-point test}
\begin{equation*}
\deep(a,b|c,d) = \longtest(\{a,b,c,d\}) \ind\{\widehat\fcal(a b | c d) > f/2\}.
\end{equation*}

\noindent\textbf{Algorithm.}
Fix $\gamma > 2$,
$0 < \eps < f/4$, $0 < \delta < 1$
and $D = 4g + 2\ln (1 + \delta) + \eps$.
Choose $c, \kappa$ so as to satisfy Proposition~\ref{proposition:concentration}.
Let $\zcal_{0}$ be the set of leaves.
The algorithm is detailed in Figure~\ref{fig:algo}.
\begin{figure*}[ht]
\framebox{
\begin{minipage}{13.2cm}
{\small \textbf{Algorithm}\\
\textit{Input:} Samples $(\s_{[n]}^i)_{i=1}^k$;\\
\textit{Output:} Tree;

\begin{itemize}
\item For $h' = 0,\ldots,h-1$,
\begin{enumerate}
\item \textbf{Deep Four-Point Test.}
Let
\begin{equation*}
\rcal_{h'} = \{q = ab|cd\ :\ \forall a,b,c,d \in \zcal_{h'}\ \text{distinct such that}\ \deep(q) = 1\}.
\end{equation*}

\item \textbf{Cherries.} Identify the cherries in $\rcal_{h'}$, that is, those pairs of vertices
that only appear on the same side of the quartet splits in $\rcal_{h'}$.
Let
\begin{equation*}
\zcal_{h'+1} = \{x_1^{(h'+1)},\ldots,x_{2^{h - (h'+1)}}^{(h'+1)}\},
\end{equation*}
be the parents of the cherries in $\zcal_{h'}$

\item \textbf{Edge Weights.} For all $x' \in \zcal_{h'+1}$,
\begin{enumerate}
\item Let $y'_1, y'_2$ be the children of $x'$.
Let $z'_1,z'_2$ be the children of $y'_1$.
Let $w'$ be any other vertex in $\zcal_{h'}$
with $\longtest(\{z'_1,z'_2,y'_2,w'\}) = 1$.

\item Let $e'_1$ be the edge between $y'_1$ and $x'$. Set
\begin{equation*}
\eweight_{e'_1} = \triplet(z'_1,z'_2;y'_2,w').
\end{equation*}

\item Repeat interchanging the role of $y'_1$ and $y'_2$.

\end{enumerate}
\end{enumerate}

\end{itemize}

}
\end{minipage}
} \caption{Tree-building algorithm.
In the deep four-point test,
internal distance estimates are used
as described in Section~\ref{section:recursive}.
} \label{fig:algo}
\end{figure*}

\subsection{Estimating the Edge Weights}
\label{section:weights}

Propositions~\ref{proposition:bias} and~\ref{proposition:variance} also
rely on edge-length estimates.
In this section, we show how this estimation is performed, assuming the tree topology is known
below $x' \in L_{h'+1}$
and edges estimates are known below level $h'$.
In Figure~\ref{fig:algo}, this procedure is used as a subroutine in the tree-building algorithm.

Let $y'_1,y'_2$ be the children of $x'$ and
let $e'_1,e'_2$ be the corresponding edges.
Let $w'$ in $L_{h'}$ be a vertex not descended from $x'$. (One should think of $w'$ as being on the same
level as on a neighboring subtree.)
Our goal is to estimate the weight of $e'_1$.
Denote by $z'_1, z'_2$ the children of $y'_1$.
(Simply set $z'_1 = z'_2 = y'_1$ if $y'_1$ is a leaf.)
Note that the internal edge of the quartet formed by $z'_1,z'_2,y'_2,w'$ is $e'_1$.
Hence, we use the standard four-point formula to compute the length of $e'_1$:
\begin{equation*}
\eweight_{e'_1}
\equiv \triplet(z'_1, z'_2; y'_2, w')
= \frac{1}{2}(\edist(z'_1,y'_2) + \edist(z'_2,w') - \edist(z'_1,z'_2) - \edist(y'_2,w')),
\end{equation*}
and
$\erho_{e'_1} = e^{-\eweight_{e'_1}}$.
Note that, with this approach, the biases at
$z'_1,z'_2,y'_2,w'$ cancel each other.
This technique was used in~\cite{DaMoRo:11a}.
\begin{proposition}[Edge-Weight Estimation]\label{proposition:weights}
Consider the setup above.
Assume that
for all $a,b \in \{z'_1,z'_2,y'_2,w'\}$ we have
\begin{equation*}
|\edist(a,b) - (\weight(a,b) + \bias(a) + \bias(b))| < \eps/2,
\end{equation*}
for some $\eps > 0$.
Then,
$|\eweight_{e'_1} - \weight_{e'_1}| < \eps$.
\end{proposition}
This result follows from a calculation similar to the proof of Proposition~\ref{proposition:bias}.

\subsection{Proof of Theorem~\ref{thm:maingaussian}}

We are now ready to prove Theorem~\ref{thm:maingaussian}.

\begin{proof}(Theorem~\ref{thm:maingaussian})
All steps of the algorithm are completed in polynomial
time in $n$ and $k$.

We argue about the correctness by induction
on the levels. Fix $\gamma > 2$.
Take $\delta > 0$, $0 < \eps < f/4$ small enough and
$c, \kappa$ large enough so that
Propositions~\ref{proposition:concentration},
\ref{proposition:bias},
\ref{proposition:variance},
\ref{proposition:weights} hold.
We divide the $\kappa \log^2 n$ samples
into $\log n$ blocks.

Assume that, using the first
$h'$ sample blocks, the topology of the
model has been correctly reconstructed
and that we have edge
estimates satisfying (\ref{eq:weightassumption})
up to level $h'$.
Assume further that we have hidden state
estimators satisfying (\ref{eq:biasguarantee})
and (\ref{eq:varguarantee}) up to level $h' - 1$
(if $h' \geq 1$).

We now use the next block of samples
which is independent of everything used
until this level.
When $h'=0$, we can use the samples directly
in the Deep Four-Point Test.  Otherwise,
we construct a linear hidden-state estimator
for all vertices on level $h'$. Propositions~\ref{proposition:bias} and \ref{proposition:variance}
ensure that conditions (\ref{eq:biasguarantee})
and (\ref{eq:varguarantee}) hold for the
new estimators. By Proposition~\ref{proposition:concentration} applied to the new
estimators
and our choice of $D = 4g + 2\ln (1 + \delta) + \eps$,
all cherries on level $h'$ appear in at least one
quartet and the appropriate quartet splits
are reconstructed. Note that the second and third terms
in $D$ account for the
bias and sampling error respectively. Once the
cherries on level $h'$ are reconstructed,
Proposition~\ref{proposition:weights}
ensures that the edge weight are estimated
so as to satisfy (\ref{eq:weightassumption}).

That concludes the induction.
\end{proof}

\subsection{Kesten-Stigum regime: Gaussian case}\label{sec:ThmNonrecon}

In this section, we derive the critical threshold for HSI
in Gaussian tree models. The section culminates
with a proof of Theorem~\ref{thm:gaussianNonRecon}
stating that TME cannot in general be achieved outside 
the KS regime without at least polynomially many
samples.

\subsubsection{Definitions}

Recall that the {\em mutual information}
between two random
vectors $\mathbf{Y}_1$ and
$\mathbf{Y}_2$ is defined as
$$
I(\mathbf{Y}_1;\mathbf{Y}_2)
= H(\mathbf{Y}_1) + H(\mathbf{Y}_2)
- H(\mathbf{Y}_1,\mathbf{Y}_2),
$$
where $H$ is the {\em entropy}, that is,
$$
H(\mathbf{Y}_1)
= - \int f_1(\mathbf{y}_1)
\log f_1(\mathbf{y}_1) d \mathbf{y}_1,
$$
assuming $\mathbf{Y}_1$ has density
$f_1$. See e.g.~\cite{CoverThomas:91}. In the Gaussian
case, if $\mathbf{Y}_1$ has covariance
matrix $\Sigma_1$, then
$$
H(\mathbf{Y}_1)
= \frac{1}{2} \log (2 \pi e)^{n_1} |\Sigma_1|,
$$
where $|\Sigma_1|$ is the determinant
of the $n_1\times n_1$ matrix $\Sigma_1$.
\begin{definition}[Solvability]
Let $\s^{(h)} _V$ be a GMRFT on
balanced tree
$$
\phy^{(h)} = (V^{(h)},E^{(h)},[n^{(h)}],\rt^{(h)};\weight^{(h)}),
$$
where $n^{(h)} = 2^h$
and $\weight^{(h)}_e = \weight > 0$
for all $e \in E^{(h)}$.
For convenience we denote
the root by $0$.
We say that the GMRFT
root state reconstruction problem with $\weight$
is {\em solvable} if
$$
\liminf_{h \to \infty} I\left(\s^{(h)}_{0};
\s^{(h)}_{[n^{(h)}]}\right) > 0,
$$
that is, if the mutual information between
the root state and leaf states remains bounded
away from $0$ as the tree size goes to
$+\infty$.
\end{definition}

\subsubsection{Threshold}

Our main result in this section is the following.
\begin{theorem}[Gaussian Solvability]
The GMRFT reconstruction problem
is solvable if and only if
$$
2 e^{-2\weight} > 1.
$$
When $2 e^{-2\weight} < 1$ then
\begin{equation}\label{e:mutualInfoDecay}
I\left(\s^{(h)}_{0}; \s^{(h)}_{[n^{(h)}]}\right) = \left[2 e^{-2\weight}\right]^h \cdot \frac{1-2 e^{-2\weight}+o(1)}{2-2 e^{-2\weight}},
\end{equation}
as $h\to\infty$.
\end{theorem}
\begin{proof}
Fix $h \geq 0$ and let
$n = n^{(h)}$,
$$
I_h
= I\left(\s^{(h)}_{0};
\s^{(h)}_{[n]}\right),
$$
$[[n]] = \{0,\ldots,n\}$, and $\rho = e^{-\weight}$.
Assume $2 \rho^2 \neq 1$.
(The case $2 \rho^2 = 1$ follows
by a similar argument which we omit.)
Denote by $\covmat^{(h)}$
and $\covmatz^{(h)}$ the covariance
matrices of $\s^{(h)}_{[n]}$
and $(\s^{(h)}_{0},
\s^{(h)}_{[n]})$ respectively.
Then
$$
I_h
= \frac{1}{2}\log \left(
\frac{|\covmat^{(h)}|}{|\covmatz^{(h)}|}
\right).
$$
Let $\ones_n$ be the all-one vector
with $n$ elements. To compute the
determinants above, we note that
each eigenvector $\vec\perp \ones_n$
of $\covmat^{(h)}$
gives an eigenvector $(0,\vec)$ of
$\covmatz^{(h)}$ with the same eigenvalue.
There are $2^h - 1$ such eigenvectors.
Further $\ones_n$ is an eigenvector of $\covmat^{(h)}$
with positive eigenvalue corresponding to the
sum of all pairwise correlation between a leaf
and all other leaves (including itself), that is,
$$
R_h
= 1 + \sum_{l=1}^h \rho^{2l} 2^{l-1}
= 1 + \rho^2\left(
\frac{(2\rho^2)^h - 1}{2\rho^2 - 1}
\right).
$$
(The other eigenvectors are obtained
inductively by noticing that each eigenvector
$\vec$ for size $2^{h-1}$ gives
eigenvectors $(\vec,\vec)$
and $(\vec,-\vec)$ for
size $2^h$.) Similarly the remaining two
eigenvectors of $\covmatz^{(h)}$
are of the form $(1, \beta \ones_n)$
with
$$
\covmatz^{(h)} (1,\beta \ones_n)'
= (1 + \beta 2^h \rho^h, (\rho^h + \beta R_h)\ones_n)'
= \lambda (1,\beta \ones_n)',
$$
whose solution is
$$
\beta_h^{\pm}
= \frac{(R_h - 1) \pm \sqrt{(R_h - 1)^2 + 4 \rho^{2h} 2^h}}{2 \rho^h 2^h},
$$
and
$$
\lambda_h^{\pm}
= 1 + \beta_h^{\pm} 2^h \rho^h.
$$
Moreover note that
\begin{eqnarray*}
\lambda_h^+
\lambda_h^-
&=& 1 + (\beta_h^+ + \beta_h^-) 2^h \rho^h
+ \beta_h^+ \beta_h^- 2^{2h} \rho^{2 h}\\
&=& 1 + (R_h - 1) - \rho^{2h} 2^h]\\
&=& R_h - (2\rho^2)^h.
\end{eqnarray*}
Hence
\begin{eqnarray*}
I_h
&=& \frac{1}{2}\log \left(
\frac{|\covmat^{(h)}|}{|\covmatz^{(h)}|}
\right)\\
&=& \frac{1}{2}\log \left(
\frac{R_h}{\lambda_h^+ \lambda_h^-}
\right)\\
&=& - \frac{1}{2}\log \left(
1 - \frac{(2\rho^2)^h}{R_h}
\right).
\end{eqnarray*}
Finally,
\begin{displaymath}
I_h
\to
\begin{cases}
0, & \text{if $2\rho^2 < 1$},\\
- \frac{1}{2}\log \left(
\frac{1}{\rho^2} - 1
\right), & \text{if $2\rho^2 > 1$},
\end{cases}
\end{displaymath}
as $h \to +\infty$ with equation~\eqref{e:mutualInfoDecay} established by a Taylor series expansion in the limit.
\end{proof}

\subsubsection{Hidden state reconstruction}

We make precise the
connection between solvability and
hidden state estimation. We are interested
in deriving good estimates of $\s^{(h)}_0$
given $\s^{(h)}_{[n]}$. Recall that the
conditional expectation $\E[\s^{(h)}_0 | \s^{(h)}_{[n]}]$
minimizes the mean squared error (MSE)~\cite{Anderson:58}.
Let $\Lambda_{[n]}^{(h)} = (\covmat^{(h)})^{-1}$.
Under the Gaussian distribution,
conditional on $\s^{(h)}_{[n]}$, the distribution
of $\s^{(h)}_0$ is Gaussian with
mean
\begin{equation}\label{e:condMeanPosterior}
\rho^h \ones_n \Lambda_{[n]}^{(h)} \s^{(h)}_{[n]}
= \frac{\rho^h}{R_h} \ones_n \s^{(h)}_{[n]},
\end{equation}
and
covariance
\begin{equation}\label{e:covarPosterior}
1 - \rho^{2h}\ones_n \Lambda_{[n]}^{(h)} \ones_n'
= 1 - \frac{(2\rho^2)^h}{R_h}
= e^{-2 I_h}.
\end{equation}
The MSE is then given by
$$
\E[(\s^{(h)}_0 - \E[\s^{(h)}_0 | \s^{(h)}_{[n]}])^2]
=\E[\var[\s^{(h)}_0 | \s^{(h)}_{[n]}]]
=e^{-2 I_h}.
$$
\begin{theorem}[Linear root-state estimation]
The linear root-state estimator
$$
\frac{\rho^h}{R_h} \ones_n \s^{(h)}_{[n]}
$$
has asymptotic MSE $< 1$
as $h \to +\infty$ if and only if
$2 e^{-2\weight} > 1$. (Note that
achieving an MSE of $1$ is trivial
with the estimator identically zero.)
\end{theorem}

The following observation explains why the
proof of our main theorem centers on the
derivation of an unbiased estimator
with finite variance.
Let $\es^{(h)}_0$ be a random variable
measurable with respect to the
$\sigma$-field generated by $\s^{(h)}_{[n]}$.
Assume that
$\E[\es^{(h)}_0|\s^{(h)}_0] = \s^{(h)}_0$,
that is, $\es^{(h)}_0$ is a conditionally
unbiased estimator of $\s^{(h)}_0$.
In particular $\E[\es^{(h)}_0] = 0$.
Then
\begin{eqnarray*}
\E[(\s^{(h)}_0 - \alpha\es^{(h)}_0)^2]
&=& \E[\E[(\s^{(h)}_0 - \alpha\es^{(h)}_0)^2|\s^{(h)}_0]]\\
&=& 1 - 2\alpha \E[\E[\s^{(h)}_0 \es^{(h)}_0|\s^{(h)}_0]] + \alpha^2 \var[\es^{(h)}_0]\\
&=& 1 - 2\alpha + \alpha^2 \var[\es^{(h)}_0],
\end{eqnarray*}
which is minimized for
$\alpha = 1/\var[\es^{(h)}_0]$. The minimum
MSE is then
$1 - 1/\var[\es^{(h)}_0]$. Therefore:
\begin{theorem}[Unbiased root-state estimator]
There exists a root-state estimator
with MSE $<1$ if and only if there exists
a conditionally unbiased root-state estimator with
finite variance.
\end{theorem}

\subsubsection{Proof of Theorem~\ref{thm:gaussianNonRecon}}

Finally in this section we establish that when $2 e^{-2\weight} < 1$ the number of samples needed for TME grows like $n^\gamma$ proving Theorem~\ref{thm:gaussianNonRecon}.

\begin{proof}(Theorem~\ref{thm:gaussianNonRecon})
The proof follows the broad approach laid out  in~\cite{Mossel:03,Mossel:04a} for establishing sample size lower bounds for phylogenetic reconstruction.
Let $\phy$ and $\tilde \phy$ be 
$h$-level balanced trees with common edge weight 
$\weight$ 
and the same vertex set differing only 
in the quartet split between the four vertices at 
graph distance 2 from the root $U=\{u_1,\ldots,u_4\}$
(that is, the grand-children of the root). 
Let $\{\s_V^{i}\}_{i=1}^k$ and $\{\tilde \s_V^{i}\}_{i=1}^k$ be $k$ i.i.d. samples from the corresponding GMRFT.

Suppose that we are given the topology of the trees below level two from the root so that all that needs to be reconstructed is the top quartet split, that is, how $U$ splits.  By the Markov property and the properties of
the multivariate Gaussian distribution, 
$\{ Y_u^i \}_{u\in U,i\in\{1,\ldots,k\}}$ with $Y_u^i=\E[\s_u^i \mid \s_{\lfloor u \rfloor}^i]$ is a sufficient statistic for the topology of the top quartet, that is, it contains all the information given by the leaf states. 
Indeed, the conditional distribution of the states at $U$
depends on the leaf states only through the condition
expectations.
To prove the impossibility of TME with high probability, we will bound the total variation distance between $\underline{Y}=\{Y_u \}_{u\in U}$ and $\tilde{\underline{Y}}=\{\tilde Y_u \}_{u\in U}$.  We have that $\underline{Y}$  is a  mean 0 Gaussian vector and using equations \eqref{e:condMeanPosterior} and \eqref{e:covarPosterior} its covariance matrix $\Sigma^*_U$ is given by
\[
(\Sigma^*_U)_{uu} = \var[Y_u] = e^{-2 I_{h-2}}=1-O((2\rho^2)^h),
\]
and
\begin{eqnarray*}
(\Sigma^*_U)_{uu'} 
&=& \cov[Y_u,X_u]\cov[X_u,X_{u'}]\cov[X_{u'},Y_{u'}]\\
&=&\frac{(2\rho^2)^{2(h-2)}}{R_{h-2}^2}(\Sigma_U)_{uu'}\\
&=&O((2\rho^2)^{2h}).
\end{eqnarray*}
where $\Sigma_U$ is the covariance matrix of 
$X_U$.  The covariance matrix of $\tilde{\underline{Y}}$ is defined similarly.  Let $\Lambda^*_U$ (resp. $\tilde\Lambda^*_U$) denote the inverse covariance matrix $(\Sigma^*_U)^{-1}$ (resp. $(\tilde\Sigma^*_U)^{-1}$).  
We note that $\Sigma^*_U$ and $\tilde\Sigma^*_U$
are close to the identity matrix and, hence, so are their inverses~\cite{HornJohnson:85}. Indeed,
with $I_U$ the $4\times 4$-identity matrix, the elements of $\Sigma^*_U-I_U$ are all $O((2\rho^2)^h)$ and, similarly for $\tilde\Sigma^*_U$, which implies that
\begin{equation}\label{e:inverseMatrixDiff}
\sup_{u,u'} |\Lambda^*_{uu'} - \tilde \Lambda^*_{uu'}| = O((2\rho^2)^h).
\end{equation}
We let $d_{\mathrm{TV}}(\cdot,\cdot)$ denote the total variation distance of two random vectors.  Note that by symmetry $|\det \Lambda^*_U | = |\det\tilde \Lambda^*_U|$ and so, with $f_{\underline{Y}}(y)$ the density function of $\underline{Y}$, the total variation distance satisfies
\begin{align*}
d_{\mathrm{TV}}(\underline{Y},\tilde{\underline{Y}})  &= \frac12 \int_{\real^4} \left|\frac{f_{\tilde{\underline{Y}}}(\underline{y})}{f_{\underline{Y}}(\underline{y})}-1\right|f_{\underline{Y}}(\underline{y}) d\underline{y} \\
&= \frac12 \int_{\real^4} \left| \exp\left[-\frac12 \underline{y}^T \tilde\Lambda^*_U \underline{y} + \frac12 \underline{y}^T \Lambda^*_U \underline{y} \right] - 1 \right| f_{\underline{Y}}(\underline{y}) d\underline{y}\\
&\leq \frac12 \int_{\real^4} \left(\exp\left[O((2\rho^2)^h \sum_{j=1}^4 y_j^2 ) \right] - 1  \right) f_{\underline{Y}}(\underline{y}) d\underline{y}\\
&\leq \frac12 \int_{\real^4} \left(\exp\left[O((2\rho^2)^h y_1^2 ) \right] - 1\right)  f_{\underline{Y}}(\underline{y}) d\underline{y}\\
&= \frac12 \left(\E\exp\left[O((2\rho^2)^h Y_{u_1}^2 ) \right] - 1\right)\\
&= O((2\rho^2)^h),
\end{align*}
where the first inequality follows from equation~\eqref{e:inverseMatrixDiff} while the second follows from an application of the AM-GM inequality and fact that the $Y_{u_i}$ are identically distributed.  The final equality follows from an expansion of equation~\eqref{e:GaussianSecondMoment}.

It follows that when $k=o((2\rho^2)^{-h})$ we can couple $\{ Y_u^i \}_{u\in U,i\in\{1,\ldots,k\}}$ and $\{ \tilde Y_u^i \}_{u\in U,i\in\{1,\ldots,k\}}$ with  probability $(1-O((2\rho^2)^h))^k$ which tends to 1.  Since they form a sufficient statistic for the top quartet, this top structure of the graph cannot be recovered with probability approaching 1.  Recalling that $n=2^h$, $\rho=e^{-\tau}$ and that if $\gamma<(2\tau)/\log 2 - 1$ then $\gmrft^{f,g}$ is not solvable with $k=n^\gamma =o((2\rho^2)^{-h})$ samples.
\end{proof}

\section{GTR Model with Unknown Rate Matrix}
\label{section:gtr}

In this section, we prove our reconstruction in the GTR
case.
We only describe the hidden-state estimator as the
other steps are the same.
We use notation similar to Section~\ref{section:gaussian}.
We denote the tree by $T=(V,E)$ with root $\rt$. The number of leaves is denoted by $n$.
Let $\nstates \geq 2$, $0 < f < g < +\infty$,
and $\phy = (V,E,[n],\rt;\weight) \in \sphy^{f,g}$.
Fix $Q \in \rates_\nstates$.
We assume that
$ 0 < g < \critks = \ln \sqrt{2}$.
We generate $k$ i.i.d.~samples $(\ss^i_{V})_{i=1}^k$ from the GTR model
$(\phy, Q)$
with state space $\states$.
Let $\eigenvec{2}$ be a second right
eigenvector of $Q$, that is, an eigenvector
with eigenvalue $-1$. We will use the notation
$\s^i_u = \eigenvec{2}_{\ss^i_u}$, for all $u\in V$ and $i=1,\ldots,k$.
We shall denote the leaves of $T$ by $[n]$.

\subsection{Estimating Rate and Frequency Parameters}\label{section:rate-matrix}

We discuss in this
section the issues involved in estimating $Q$ and its eigenvectors using data at the leaves.
For the purposes of our algorithm we need only estimate
the first left eigenvector and the second right eigenvector.
Let $\pi$ be the stationary distribution of $Q$ (first left eigenvector) and denote
$\Pi = \diag(\pi)$.
Let
\begin{equation*}
\eigenvec{1}, \eigenvec{2}, \ldots, \eigenvec{\nstates},
\end{equation*}
be the right eigenvectors of $Q$ corresponding respectively to eigenvalues
\begin{equation*}
0 = \eigenval{1}> \eigenval{2}\geq \ldots \geq \eigenval{\nstates}.
\end{equation*}
Because of the reversibility assumption, we can choose the eigenvectors to be
orthonormal with respect to the inner product,
\begin{equation*}
\langle \evr , \evr' \rangle_\pi = \sum_{i\in\states} \pi_i \evr_i \evr'_i
\end{equation*}
In the case of multiplicity of eigenvalues this description may not be unique.
\begin{proposition}\label{prop:rateMatrixApprox}
There exists $\kappa(\epsilon,\eigenratio,Q)$ such that given $\kappa\log n$ samples there exist estimators $\hat\pi$ and $\eeigenvec{2}$ such that 
\begin{equation}\label{e:piHatApprox}
\parallel \pi-\hat\pi\parallel\leq \epsilon,
\end{equation}
and
\begin{equation}\label{e:eigenapprox}
\eeigenvec{2} = \sum_{l=1}^{\nstates} \alpha_l \evr^l,
\end{equation}
where $|\alpha_2 - 1| \leq \eps$ and $|\frac{\alpha_l}{\alpha_2} | < \eigenratio$ for $l \geq 3$, (for some choice of $\evr^l$ if the second eigenvalue has multiplicity greater than 1).
\end{proposition}

\paragraph{Estimates}
Let $\widehat F$ denote the empirical joint distribution at leaves $a$ and $b$ as a $\nstates\times\nstates$ matrix.
(We use an extra sample block for this estimation.)
To estimate $\pi$ and $\eigenvec{2}$, our first task is to find two leaves that are sufficiently close
to allow accurate estimation.
Let $a^*,b^* \in [n]$ be two leaves with minimum log-det distance
\begin{equation*}
(a^*, b^*) \in \arg\min \left\{-\log\det\widehat{F}^{ab}\ :\ (a,b) \in [n]\times [n]\right\}.
\end{equation*}
Let
\begin{equation*}
F = F^{a^* b^*},
\end{equation*}
and consider the symmetrized correlation matrix
\begin{equation*}
\widehat{F}^\dagger = \frac{1}{2}(\widehat{F}^{a^* b^*} + (\widehat{F}^{a^* b^*})^\top).
\end{equation*}
Then we estimate $\pi$ from
\begin{equation*}
\hat\pi_\upsilon = \sum_{\upsilon' \in \states} \widehat{F}^\dagger_{\upsilon \upsilon'},
\end{equation*}
for all $\upsilon \in \states$. Denote $\widehat\Pi = \diag(\hat\pi)$.
By construction, $\hat\pi$ is a probability distribution.
Let $\varphi = \weight(a^*, b^*)$ and define $G$ to be the symmetric matrix
\begin{equation*}
G = \Pi^{-1/2}F\Pi^{-1/2} = \Pi^{-1/2}(\Pi e^{\varphi Q})\Pi^{-1/2} = \Pi^{1/2} e^{\varphi Q}\Pi^{-1/2}.
\end{equation*}
Then denote the right eigenvectors
of $G$ as
\begin{equation*}
\edgeeigenvec{1} = \Pi^{1/2} \eigenvec{1}, \edgeeigenvec{2} = \Pi^{1/2}\eigenvec{2}, \ldots, \edgeeigenvec{\nstates} = \Pi^{1/2}\eigenvec{\nstates},
\end{equation*}
with corresponding eigenvalues
\begin{equation*}
1 = \edgeeigenval{1}_{(a^*,b^*)} = e^{\varphi\eigenval{1}} > \edgeeigenval{2}_{(a^*,b^*)}  = e^{\varphi\eigenval{2}}
\geq \cdots \geq \edgeeigenval{\nstates}_{(a^*,b^*)}  = e^{\varphi\eigenval{\nstates}},
\end{equation*}
orthonormal with respect to the Euclidean inner product.
Note that $\edgeeigenval{2}_{(a^*,b^*)}   < e^{-f}$
and that $\eigenvec{1}$ is the all-one vector.  Assuming $\hat\pi > 0$, define
\begin{equation*}
\widehat{G} = \widehat{\Pi}^{-1/2}\widehat{F}^\dagger\widehat{\Pi}^{-1/2}.
\end{equation*}
which we use to estimate the eigenvectors and eigenvalues of $Q$.  Since $\widehat{G}$ is real symmetric,
it has $\nstates$ real eigenvalues
$
\eedgeeigenval{1}> \eedgeeigenval{2}\geq \ldots \geq \eedgeeigenval{\nstates}.
$
with a corresponding orthonormal basis
$
\eedgeeigenvec{1}, \eedgeeigenvec{2}, \ldots, \eedgeeigenvec{\nstates}.
$
It can be checked that, provided $\widehat{G} > 0$,
we have $1 = \eedgeeigenval{1} >  \eedgeeigenval{2} $.
We use
\begin{equation*}
\eeigenvec{2} = \widehat{\Pi}^{-1/2} \eedgeeigenvec{2}.
\end{equation*}
as our estimate of the
``second eigenvector'' and $\eedgeeigenval{2} $ as our estimate of the second eigenvalue of the channel.

\paragraph{Discussion}
The sensitivity of eigenvectors is somewhat delicate~\cite{HornJohnson:85}.
With sufficiently many samples ($k=\kappa \log n$ for large enough $\kappa$)
the estimator $\widehat{G}$ will approximate $G$ within any constant tolerance.
When the second eigenvalue is distinct from the third one our estimate will satisfy
\eqref{e:eigenapprox} provided $\kappa$ is large enough.

If there are multiple second eigenvectors the vector $\eeigenvec{2}$ may not exactly be an estimate of $\eigenvec{2}$
since indeed the second eigenvalue is not uniquely defined:
using classical results (see e.g.~\cite{GolubVanLoan:96}) it can be shown
that $\eeigenvec{2} $ is close to a combination of eigenvectors
with eigenvalues equal to $\edgeeigenval{2} $.  Possibly after passing to a different basis of eigenvectors
$\eigenvec{1}, \eigenvec{2}, \ldots, \eigenvec{\nstates}$, we still have that equation \eqref{e:eigenapprox} holds.  By standard large deviations estimate this procedure satisfies Proposition~\ref{prop:rateMatrixApprox} when $\kappa$ is large enough.

\begin{remark}
This procedure provides arbitrary accuracy as $\kappa$ grows, however, for fixed $\kappa$ it will not in general go to 0 as
$n$ goes to infinity as the choice of $a^*,b^*$ may bias the result.  An error of size $O(1/\sqrt{k})$ may be obtained by taking all
pairs with log-det distance below some small threshold (say $4g$), randomly picking such a pair $a',b'$ and estimating the matrix
$\widehat{G}$ using $a',b'$.

We could also have estimated $\hat\pi$ by taking the empirical distribution of the states at one of the vertices or indeed the empirical distribution over all vertices.
\end{remark}

\subsection{Recursive Linear Estimator}

As in the Gaussian case, we build a recursive
linear estimator.
We use notation similar to Section~\ref{section:gaussian}.
Let $K = \kappa \log n$ be the size of each block.
We let $\ss_V$ be a generic sample from
the GRT model independent of
everything else, and we define
$\s_u = \eeigenvec{2}_{\ss_u}$
for all $u\in V$.
We let $(\ss_{[n]}^i)_{i=1}^K$
be a block of independent samples at the leaves,
and we set $\s_u^\ell = \eeigenvec{2}_{\ss_u^\ell}$,
for all $u\in V$ and $\ell = 1,\ldots,K$.
For a node $u \in V$, we let $\desc{u}$ be the
leaves below $u$ and $\sb{u}$, the corresponding
state.
Let $0 < \delta < 1$ (small) and $c > 1$ (big) be constants to be defined later.

\paragraph{Linear estimator}
We build a linear estimator for each of the vertices recursively from the leaves.
Let $x\in V-[n]$ with children (direct descendants)
$y_1,y_2$.
Assume that the topology of the tree rooted at $x$
has been correctly reconstructed.
Assume further that we have constructed linear
estimators
$$S_u \equiv \lin_u(\sb{u})$$
of $\s_u$, for all $u \in V$ below $x$.
We use the convention that
$$
\lin_u(\sb{u}) = \s_u
$$
if $u$ is a leaf.
We let $\lin_x$ be a linear
combination of the form
\begin{equation}\label{eq:linearRobust}
S_x \equiv \lin_x(\sb{x}) =
\omega_{y_1} \lin_{y_1}(\sb{y_1}) +
\omega_{y_2} \lin_{y_2}(\sb{y_2}),
\end{equation}
where the $\omega$'s are chosen below.

\paragraph{Recursive conditions}
Assume that we have linear
estimators $\lin_{u}$ for all $u$ below $x$ satisfying
\begin{equation}\label{eq:biasguaranteeRobust}
\expec[S_{u} \,|\, \ss_{u}]
= \sum_{l=1}^{\nstates}
\bbias^l(u) \eigenvec{l}_{\ss_{u}},
\end{equation}
for some $\bbias^l(u)$ such that  $|\bbias^2(u) - 1| < \delta$ and $| \bbias^l(u) / \bbias^2(u)|<\eigenratio$
for $l = 3,\ldots,\nstates$.  Note that no condition is placed on $\bbias^1(u)$.  Further for all $i\in\states$
\begin{equation}\label{eq:expmomguaranteeRobust}
\Gamma_{u}^i(\zeta)
\leq \zeta \expec[S_{u}\,|\,\ss_{u} = i] + c \zeta^2,
\end{equation}
where as before
$$
\Gamma_u^i(\zeta) \equiv \ln\expec[\exp(\zeta S_u)\,|\,\ss_u = i].
$$
Observe that these conditions
are satisfied at the leaves.
Indeed, for $u \in [n]$ one has
$S_{u}= \eeigenvec{2}_{\ss_{u}}
= \sum_{l=1}^{\nstates} \alpha_l \eigenvec{l}_{\ss_u}$ and therefore
$\expec[S_{u}\,|\,\ss_{u}]
=   \sum_{l}^{\nstates} \alpha_l \eigenvec{l}_{\ss_u}$
and $\Gamma_{u}^i(\zeta) =\zeta \expec[S_{u}\,|\,\ss_{u} = i] $.
We now seek to construct $S_x$ so that it in turn satisfies the same conditions.

Moreover we assume we have a priori estimated edge weights $\eweight_{e}$ for all $e$ below $x$
such that for $\eps > 0$
we have that
\begin{equation}\label{eq:weightassumptionRobust}
|\eweight_{e} - \weight_{e}| < \eps.
\end{equation}
Let $\hat\theta_{e} = e^{-\eweight_{e}}$.

\paragraph{First eigenvalue adjustment}
As discussed above, because we cannot estimate exactly
the second eigenvector, our estimate $\eeigenvec{2}$ may contain components of other eigenvectors.
While eigenvectors $\eigenvec{3}$ through $\eigenvec{\nstates}$ have smaller eigenvalues and will
thus decay in importance as we recursively construct our estimator, the presence of a component
in the direction of the first eigenvalue poses greater difficulties.
However, we note that $\eigenvec{1}$ is identically 1.  So to remove the effect
of the first eigenvalue from equation \eqref{eq:biasguaranteeRobust} we subtract the empirical mean of $S_u$,
\[
\bar S_u =  \frac{1}{K} \sum_{\ell=1}^K S^\ell_{u}.
\]
As $\langle \pi , \eigenvec{l}\rangle = 0$ for $l=2,\ldots,\nstates$ and $\eigenvec{1}\equiv 1$ we have that
$\expec S_u = \bbias^1(u)$
from (\ref{eq:biasguaranteeRobust})
and hence the following proposition follows from standard large deviations estimates.
\begin{proposition}[Concentration of Empirical Mean]\label{proposition:conMean}
For $u \in V$, $\eps' > 0$ and $\gamma>0$, suppose that conditions \eqref{eq:biasguaranteeRobust} and \eqref{eq:expmomguaranteeRobust} hold for some $\delta,\eps$ and $c$.  Then there exists $\kappa = \kappa(\eps', c,\gamma,\delta,\eps) > 0$ such that, when we have $K \geq \kappa \log n$
then
\begin{equation*}
|\bar S_u - \bbias^1(u)| < \eps',
\end{equation*}
with probability at least $1 - O(n^{-\gamma})$.
\end{proposition}
\begin{proof}
Let $\eps_\pi > 0$.
By Chernoff's bound,
of the $K$ samples, $\widehat K_i$ are such that
$\ss_u^\ell = i$ where
$$
\left|\frac{\widehat{K}_i}{K} - \pi_i\right|
\leq \eps_\pi,
$$
except with inverse polynomial probability,
given that $\kappa$ is large enough.
By
\eqref{eq:biasguaranteeRobust} and
\eqref{eq:expmomguaranteeRobust},
we have
$$
\E[e^{\zeta(S_u - \bbias^1(u))}|\ss_u = i]
\leq \zeta \E[(S_u - \bbias^1(u))|\ss_u = i]
+ c \zeta^2,
$$
where
$$
\left|
\E[(S_u - \bbias^1(u))|\ss_u = i]
\right|
=
\left|
\sum_{l=2}^{\nstates}
\bbias^l(u) \eigenvec{l}_{i}
\right|
\leq
(1+\delta)(1+q\eigenratio)\max_j 1/\sqrt{\pi_j}
\equiv \Upsilon.
$$
Let $\eps_\Gamma > 0$.
Choosing $\zeta = \frac{\eps_\Gamma}{2c}$
in Markov's inequality for $e^{\zeta (S_u - \bbias^1(u))}$
gives that the average of $S_u^\ell - \bbias^1(u)$
over the samples with $\ss_u^\ell = i$ is
within $\eps_\Gamma$ of
$\sum_{l=2}^{\nstates}
\bbias^l(u) \eigenvec{l}_{i}$
except with probability at most $e^{-\eps_\Gamma^2
K(\pi_i - \eps_\pi)/4c}
= 1/\poly(n)$ for $\kappa$ large enough
and $\eps_\pi$ small enough.
Therefore, in that case,
\begin{eqnarray*}
\left|
\frac{1}{K} \sum_{\ell=1}^K (S^\ell_{u} - \bbias^1(u))
\right|
&\leq&
q\eps_\Gamma + \eps_\pi[\Upsilon + \eps_\Gamma] < \eps',
\end{eqnarray*}
for $\eps_\pi, \eps_\Gamma$ small enough,
where we used $\langle \pi , \eigenvec{l}\rangle = 0$ for $l=2,\ldots,\nstates$.
\end{proof}

For $\alpha=1,2$, using the Markov property we have the following important conditional moment identity which we will use to relate the bias at $y_\alpha$ to the bias at $x$,
\begin{align}\label{e:conditionalMeanRobust}
\expec \left( S^\ell_{y_\alpha} - \bbias^1(y_\alpha) \mid \ss_{x} = i \right) &=  \sum_{l=2}^\nstates \sum_{j=1}^\nstates \bbias^l(y_\alpha) M^{e_\alpha}_{ij}  \eigenvec{l}_{j}\nonumber \\
&=  \sum_{l=2}^\nstates  \bbias^l(y_\alpha)  \edgeeigenval{l}_{e_\alpha} \eigenvec{l}_{i},
\end{align}
where we used the fact that the $\eigenvec{l}$'s are eigenvectors of $M^{e_\alpha}_{ij}$ with eigenvectors $\edgeeigenval{l}_e = \exp(-\eigenval{l}\weight_e)$.

\paragraph{Procedure}
We first define a procedure for estimating the path length
(that is, the sum of edge weights)
between a pair of vertices $u_1$ and $u_2$ including the bias.  For $u_1, u_2 \in V$ with common ancestor $v$ we define
\begin{equation*}
\edist(u_1,u_2) = -\ln\left(\frac{1}{K} \sum_{\ell=1}^K \left( S^\ell_{u_1} - \bar S_{u_1}\right)  \left( S^\ell_{u_2} - \bar S_{u_2}\right) \right).
\end{equation*}
This estimator differs from Section \ref{section:recursive} in that we subtract the empirical means to remove the effect of the first eigenvalue.  Using the fact that $\sum_{\ell=1}^k  S^\ell_{u_1} - \bar S_{u_1} = 0$ and Proposition \ref{proposition:conMean} we have that with probability at least $1 - O(n^{-\gamma})$
\begin{eqnarray*}
&&\frac{1}{K} \sum_{\ell=1}^K \left( S^\ell_{u_1} - \bar S_{u_1}\right)  \left( S^\ell_{u_2} - \bar S_{u_2}\right) \\
&& \qquad =  \frac{1}{K} \sum_{\ell=1}^K \Big[\left( S^\ell_{u_1} - \bbias^1(u_1) \right)  \left( S^\ell_{u_2} - \bbias^1(u_2) \right)\\
&&\qquad\qquad +  \left( \bar S^\ell_{u_1} - \bbias^1(u_1) \right)  \left(\bar S^\ell_{u_2} - \bbias^1(u_2) \right)\Big]\nonumber \\
&& \qquad \leq  \frac{1}{K} \sum_{\ell=1}^K \left( S^\ell_{u_1} - \bbias^1(u_1) \right)  \left( S^\ell_{u_2} - \bbias^1(u_2) \right) + (\eps')^2,
\end{eqnarray*}
and similarly the other direction so,
\begin{eqnarray}\label{e:empiricalMeanSub}
&&\bigg|\frac{1}{K} \sum_{\ell=1}^K \left( S^\ell_{u_1} - \bar S_{u_1}\right)  \left( S^\ell_{u_2} - \bar S_{u_2}\right) \nonumber\\
&& \qquad - \frac{1}{K} \sum_{\ell=1}^K \left( S^\ell_{u_1} - \bbias^1(u_1) \right)  \left( S^\ell_{u_2} - \bbias^1(u_2) \right) \bigg| \leq (\eps')^2.
\end{eqnarray}
It follows that $\edist(u_1,u_2)$ is an estimate of the length between $u_1$ and $u_2$ including bias since
\begin{align}\label{e:covariance1}
&\expec \left[\left( S^\ell_{u_1} - \bbias^1(u_1) \right)  \left( S^\ell_{u_2} - \bbias^1(u_2) \right)\right]\nonumber\\
& \qquad = \sum_{i\in \states } \pi_i  \expec \left( S^\ell_{u_1} - \bbias^1(u_1) \mid \ss_{v} = i \right) \expec \left( S^\ell_{u_2} - \bbias^1(u_2) \mid \ss_{v} = i \right) \nonumber\\
&\qquad = \sum_{i\in \states } \pi_i \left (\sum_{l=2}^\nstates  \bbias^l(u_1)  \edgeeigenval{l}_{(v,u_1)} \eigenvec{l}_{j}\right)\left( \sum_{l=2}^\nstates  \bbias^l(u_2)  \edgeeigenval{l}_{(v,u_2)} \eigenvec{l}_{j}\right) \nonumber\\
&\qquad=  \bbias^2(u_1)  \edgeeigenval{2}_{(v,u_1)}  \bbias^2(u_2)  \edgeeigenval{2}_{(v,u_2)} + O(\eigenratio)\nonumber\\
&\qquad=  \bbias^2(u_1) \bbias^2(u_2) e^{-\tau(u_1,u_2)} + O(\eigenratio),
\end{align}
where line 2 follows from equation \eqref{e:conditionalMeanRobust}. Above we also used the recursive assumptions  and the fact that $\sum_{i\in \states} \pi_i (\eigenvec{2}_i)^2=1$.
We will use the estimator $\edist(u,v) $ to estimate $\bias(u) = -\ln \bbias^2(u)$.  Given the previous setup, we choose the weights $\omega_{y_\alpha}$, $\alpha=1,2$, as follows:
\begin{enumerate}
\item {\bf Estimating the Biases.}
If $y_1, y_2$ are leaves, we let $\ebias(y_\alpha) = 0$, $\alpha=1,2$.
Otherwise, let $z_{\alpha 1},z_{\alpha 2}$ be the children of $y_{\alpha}$. We then compute
\begin{equation*}
\ebias(y_1) = \frac{1}{2}(\edist(y_1, z_{21}) + \edist(y_1, z_{22}) - \edist(z_{21}, z_{22}) - 2\eweight_{e_1} - 2\eweight_{e_2}),
\end{equation*}
And similarly for $y_2$.
Let $\ebbias^2(y_\alpha) = e^{-\ebias(y_\alpha)}$, $\alpha=1,2$.

\item {\bf Minimizing the Variance.}
Set $\omega_{y_\alpha}, \alpha=1,2$ as
\begin{equation}\label{eq:solutionRobust}
\omega_{y_\alpha} = \frac{\ebbias^2(y_\alpha)\edgeeigenval{2}_{e_\alpha}}{(\ebbias^2(y_1))^2(\edgeeigenval{2}_{e_1})^2 + (\ebbias^2(y_2))^2(\edgeeigenval{2}_{e_2})^2},
\end{equation}
the solution of the following optimization problem:
\begin{equation}\label{eq:optimizationRobust}
\min\{\omega_{y_1}^2 + \omega_{y_2}^2 \ :\ \omega_{y_1} \ebbias^2(y_1) \edgeeigenval{2}_{e_1} +  \omega_{y_2} \ebbias^2(y_2) \edgeeigenval{2}_{e_2} = 1,\ \omega_{y_1},\omega_{y_2} > 0 \}.
\end{equation}
The constraint above guarantees that the bias condition (\ref{eq:biasguaranteeRobust}) is satisfied when we set
\begin{equation*}
\lin_x(\sb{x}) = \omega_{y_1} \lin_{y_1}(\sb{y_1})
+ \omega_{y_2} \lin_{y_2}(\sb{y_2}).
\end{equation*}

\end{enumerate}

\paragraph{Bias and Exponential Moment}
We now prove (\ref{eq:biasguaranteeRobust}) and (\ref{eq:expmomguaranteeRobust}) recursively assuming
(\ref{eq:weightassumptionRobust}) is satisfied.
Assume the setup of the previous paragraph. We already argued that
(\ref{eq:biasguaranteeRobust}) and (\ref{eq:expmomguaranteeRobust}) are satisfied at the leaves.
Assume further that they are satisfied for all descendants of $x$.
We first show that the $\edist$-quantities are concentrated.
\begin{proposition}[Concentration of Internal Distance Estimates]\label{proposition:concentrationRobust}
For all $\eps > 0$, $\gamma > 0$, $0 < \delta < 1$ and
$c > 0$, there are $\kappa = \kappa(\eps,\gamma,\delta,c) > 0$, $\eigenratio = \eigenratio(\eps,\gamma,\delta,c) > 0$ such that,
with probability at least $1 - O(n^{-\gamma})$, we have
\begin{equation*}
|\edist(u,v) - (\weight(u,v) + \bias(u) + \bias(v))| < \eps,
\end{equation*}
for all $u,v \in \{y_1,y_2,z_{11},z_{12},z_{21},z_{22}\}$ where $z_{\alpha 1}, z_{\alpha 2}$ are the children
of $y_\alpha$.
\end{proposition}
\begin{proof}
This proposition is proved similarly to Proposition \ref{proposition:concentration} by establishing concentration of
$
\frac{1}{K} \sum_{\ell=1}^K \widetilde{S}^\ell_u
\widetilde{S}^\ell_v,
$
where $\widetilde{S}^\ell_u = S^\ell_{u} - \bbias^1(u)$,
around its mean which is approximately
$
e^{-\weight(u,v) - \bias(u) - \bias(v)}
$
by
equation \eqref{e:covariance1}. The only difference
with Proposition~\ref{proposition:concentration} is
that, in this non-Gaussian case, we must
estimate the exponential moment
directly using \eqref{eq:expmomguaranteeRobust}.
We use an argument of~\cite{PeresRoch:11,Roch:10}.

Let $\zeta > 0$.
Let $N$ be a standard normal. Using that
$\E[e^{\alpha N}] = e^{\alpha^2/2}$ and applying
\eqref{eq:biasguaranteeRobust} and
\eqref{eq:expmomguaranteeRobust},
\begin{eqnarray*}
\E[e^{\zeta \widetilde{S}_u \widetilde{S}_v}|\ss_{\{u,v\}}]
&\leq&
\E[e^{(\zeta \widetilde{S}_u) \E[\widetilde{S}_v|\ss_v]
+ c (\zeta \widetilde{S}_u)^2 }|\ss_{\{u,v\}}]\\
&=& \E[e^{\zeta \widetilde{S}_u \E[\widetilde{S}_v|\ss_v]
+ \sqrt{2c} \zeta \widetilde{S}_u N }|\ss_{\{u,v\}}]\\
&\leq& \E[e^{(\zeta \E[\widetilde{S}_v|\ss_v]
+ \sqrt{2c} \zeta N)\E[\widetilde{S}_u|\ss_u]
+  c(\zeta \E[\widetilde{S}_v|\ss_v]
+ \sqrt{2c} \zeta N)^2}|\ss_{\{u,v\}}].
\end{eqnarray*}
We factor out the constant term and apply
Cauchy-Schwarz on the linear and quadratic
terms in $N$
\begin{eqnarray*}
&&\E[e^{\zeta \widetilde{S}_u \widetilde{S}_v}|\ss_{\{u,v\}}]\\
&& \quad \leq
e^{\zeta \E[\widetilde{S}_u|\ss_u] \E[\widetilde{S}_v|\ss_v]} e^{c\zeta^2\Upsilon^2} \E[e^{4c^2 \zeta^2 N^2}]^{1/2}\\
&& \qquad \times\E\left[
e^{2(\sqrt{2c} \zeta \E[\widetilde{S}_u|\ss_u]
+ 2 c \sqrt{2c} \zeta^2 \E[\widetilde{S}_v|\ss_v])N}
|\ss_{\{u,v\}} \right]^{1/2}
\\
&& \quad \leq
e^{\zeta \E[\widetilde{S}_u|\ss_u] \E[\widetilde{S}_v|\ss_v]}
e^{c\zeta^2\Upsilon^2}
\frac{1}{(1 - 8c^2\zeta^2)^{1/4}}
e^{2c\Upsilon^2\zeta^2 (1 + 2c\zeta)^2}\\
&& \quad = 1
+ \zeta \E[\widetilde{S}_u \widetilde{S}_v|\ss_{\{u,v\}}]
+ \Upsilon' \zeta^2
+ O(\zeta^3),
\end{eqnarray*}
as $\zeta \to 0$,
where $\Upsilon$ was defined in the proof
of Proposition~\ref{proposition:conMean}
and $\Upsilon' > 0$ is a constant depending on $\Upsilon$
and $c$. Taking expectations and expanding
\begin{eqnarray*}
e^{-\zeta(\E[\widetilde{S}_u \widetilde{S}_v] + \eps)}
\E[e^{\zeta \widetilde{S}_u \widetilde{S}_v}]
= 1 - \eps\zeta + \Upsilon'\zeta^2 + O(\zeta^3) < 1,
\end{eqnarray*}
for $\zeta$ small enough, independently of $n$.
Applying Markov's inequality
gives the result.
\end{proof}

\begin{proposition}[Recursive Linear Estimator: Bias]\label{proposition:biasRobust}
Assuming (\ref{eq:biasguaranteeRobust}), (\ref{eq:expmomguaranteeRobust}), and (\ref{eq:weightassumptionRobust}) hold
for some $\eps > 0$ that is small enough,
we have
\begin{equation*}
\expec[S_{x} \,|\, \ss_x] = \sum_{l=1}^{\nstates} \bbias^l(x) \eigenvec{l}_{\ss_{x}},
\end{equation*}
for some $\bbias^l(x)$ such that  $|\bbias^2(x) - 1| < \delta$ and $|  \bbias^l(x) / \bbias^2(x)|<\eigenratio$ for $l = 3,\ldots,\nstates$.
\end{proposition}
\begin{proof}
We first show that the biases at $y_1, y_2$ are accurately estimated. Applying a similar  proof to that of Proposition \ref{proposition:bias} (using Proposition~\ref{proposition:concentrationRobust}  in place of Proposition~\ref{proposition:concentration}) we have that
\begin{equation*}
|\ebias(y_1) - \bias(y_1)| \leq O(\eps + \eigenratio).
\end{equation*}
The same inequality holds for $y_2$.
Taking $\eps, \delta$ small enough, our previous bounds on $\bbias$, $\theta$ and
their estimates, we derive from equation \eqref{eq:solutionRobust} that $\omega_{y_\alpha} = \Theta(1)$, $\alpha=1,2$ with high probability.  We now calculate the bias at $x$ to be,
\begin{eqnarray*}
\expec[S_{x} \,|\, \ss_{x} = i]
&=& \expec[\omega_{y_1}S_{y_1} + \omega_{y_2}S_{y_2} \,|\, \ss_{x} = i]\\
&=& \sum_{\alpha=1,2} \omega_{y_\alpha} \sum_{l=1}^\nstates  \bbias^l(y_\alpha)  \edgeeigenval{l}_{e_\alpha} \eigenvec{l}_{j}\\
&=&   \sum_{l=1}^\nstates \left( \omega_{y_1} \bbias^l(y_1)  \edgeeigenval{l}_{e_1}   +   \omega_{y_2} \bbias^l(y_2)  \edgeeigenval{l}_{e_2}  \right)\eigenvec{l}_{j}\\
&\equiv  &   \sum_{l=1}^\nstates  \bbias^l(x)  \eigenvec{l}_{j}\\
\end{eqnarray*}
where we used equation \eqref{e:conditionalMeanRobust} on line 2.
Observe that since $\omega_{y_1}, \omega_{y_2}$ are positive and $0< \edgeeigenval{l}_{e_\alpha}   \leq \edgeeigenval{2}_{e_\alpha}$ for $l\geq 3$,
\begin{eqnarray*}
\left| \frac{\bbias^l(x)}{\bbias^2(x)} \right|
&=&
\left|  \frac{   \omega_{y_1} \bbias^l(y_1)  \edgeeigenval{l}_{e_1}   +   \omega_{y_2} \bbias^l(y_2)  \edgeeigenval{l}_{e_2}   }{   \omega_{y_1} \bbias^2(y_1)  \edgeeigenval{2}_{e_1}   +   \omega_{y_2} \bbias^2(y_2)  \edgeeigenval{2}_{e_2}   } \right|\\
&\leq&
\left|  \frac{   \omega_{y_1}  \eigenratio\bbias^2(y_1)  \edgeeigenval{2}_{e_1}   +   \omega_{y_2} \eigenratio \bbias^2(y_2)  \edgeeigenval{2}_{e_2}   }{   \omega_{y_1} \bbias^2(y_1)  \edgeeigenval{2}_{e_1}   +   \omega_{y_2} \bbias^2(y_2)  \edgeeigenval{2}_{e_2}   } \right|\\
&=&  \eigenratio.
\end{eqnarray*}

Applying the bounds on $\omega_{y_\alpha}$ and $\ebias(y_\alpha)$ for $\alpha=1,2$ we have that
\begin{eqnarray*}
\bbias^2(x)
&=& \omega_{y_1} \bbias^2(y_1) \edgeeigenval{2}_{e_1} +  \omega_{y_2} \bbias^2(y_2) \edgeeigenval{2}_{e_2}\\
&=& \omega_{y_1} e^{-\bias(y_1)} \edgeeigenval{2}_{e_1} +  \omega_{y_2} e^{-\bias(y_2)} \edgeeigenval{2}_{e_2}\\
&\leq& \omega_{y_1} e^{-\ebias(y_1) +O(\eps+ \eigenratio)}
(\eedgeeigenval{2}_{e_1} + O(\eps+ \eigenratio))\\
&&\quad  +  \omega_{y_2} e^{-\ebias(y_2)+O(\eps+ \eigenratio)} (\eedgeeigenval{2}_{e_2}  +O(\eps+ \eigenratio))\\
&=& (\omega_{y_1} \ebbias^2(y_1) \eedgeeigenval{2}_{e_1} +  \omega_{y_2} \ebbias^2(y_2) \eedgeeigenval{2}_{e_2}) + O(\eps+ \eigenratio)\\
&=& 1 +  O(\eps+ \eigenratio).
\end{eqnarray*}
Choosing $\eps$ and $\rho$ small enough, it satisfies $|\bbias^2(x) - 1| < \delta$.
\end{proof}
\begin{proposition}[Recursive Linear Estimator: Exponential Bound]\label{proposition:exponentialRobust}
There is $c > 0$ such that,
assuming (\ref{eq:biasguaranteeRobust}), (\ref{eq:expmomguaranteeRobust}), and (\ref{eq:weightassumptionRobust}) hold,
we have for all $i\in\states$
\begin{equation*}
\Gamma_{x}^i(\zeta)
\leq \zeta \expec[S_{x}\,|\,\ss_{x} = i] + c \zeta^2.
\end{equation*}
\end{proposition}
\begin{proof}
We use the following lemma suitably generalized from~\cite{PeresRoch:11,Roch:10}.
\begin{lemma}[Recursion Step]\label{lemma:stepRobust}
Let $M = e^{\weight Q}$ as above with eigenvectors
$$
\eigenvec{1},\eigenvec{2},\ldots,\eigenvec{\nstates},
$$
with corresponding eigenvalues $1=e^{\eigenval{1}}\geq    \ldots \geq e^{\eigenval{\nstates}}$.  Let $b_2,\ldots,b_\nstates$ we arbitrary constants with $|b_i| < 2$.
Then there is $c' > 0$ depending
on $Q$ such that for all $i\in\states$
\begin{equation*}
F(x) \equiv \sum_{j \in \states} M_{i j} \exp\left(  x  \sum_{l=2}^\nstates  b_l  \eigenvec{l}_{j} \right)
\leq \exp\left(x  \sum_{l=2}^\nstates \lambda_l b_l  \eigenvec{l}_{i}  + c' x^2   \right)
\equiv G(x),
\end{equation*}
for all $x \in \real$.
\end{lemma}
We have by the Markov property and Lemma~\ref{lemma:stepRobust} above,
\begin{eqnarray*}
\Gamma_x^i(\zeta)
&=& \ln\expec\left[\exp\left(\zeta \sum_{\alpha=1,2}S_{y_\alpha} \omega_{y_\alpha} \right)\,|\,\ss_x = i\right]\\
&=& \sum_{\alpha=1,2}\ln\expec\left[\exp\left(\zeta S_{y_\alpha} \omega_{y_\alpha} \right)\,|\,\ss_x = i\right]\\
&=& \sum_{\alpha=1,2}\ln\left(\sum_{j\in\states} M^{e_\alpha}_{ij}
\expec\left[\exp\left(\zeta S_{y_\alpha} \omega_{y_\alpha} \right)\,|\,\ss_{y_\alpha} = j\right]\right)\\
&=& \sum_{\alpha=1,2}\ln\left(\sum_{j\in\states} M^{e_\alpha}_{ij}
\exp\left(\Gamma_{y_\alpha}^j\left(\zeta \omega_{y_\alpha} \right)\right)\right)\\
&\leq& \sum_{\alpha=1,2}\ln\left(\sum_{j\in\states} M^{e_\alpha}_{ij}
\exp\left(
\zeta \omega_{y_\alpha} \expec[S_{y_\alpha}\,|\,\ss_{y_\alpha} = j] + c \zeta^2\omega_{y_\alpha}^2
\right)\right)\\
&=& c\zeta^2 \sum_{\alpha=1,2} \omega_{y_\alpha}^2 + \sum_{\alpha=1,2}\ln\left(\sum_{j\in\states} M^{e_\alpha}_{ij}
\exp\left(\zeta \omega_{y_\alpha}    \sum_{l=1}^{\nstates} \bbias^l(y_\alpha) \eigenvec{l}_{j}   \right)\right)\\
&=& c\zeta^2 \sum_{\alpha=1,2} \omega_{y_\alpha}^2
+\zeta \sum_{\alpha=1,2} \bbias^1(y_\alpha)\omega_{y_\alpha}\\
&&\quad + \sum_{\alpha=1,2}\ln\left(\sum_{j\in\states} M^{e_\alpha}_{ij}
\exp\left(\zeta \omega_{y_\alpha}    \sum_{l=2}^{\nstates} \bbias^l(y_\alpha) \eigenvec{l}_{j}   \right)\right)\\
&\leq& c\zeta^2 \sum_{\alpha=1,2} \omega_{y_\alpha}^2
+  \zeta  \sum_{\alpha=1,2} \omega_{y_\alpha}  \sum_{l=1}^{\nstates}
\edgeeigenval{l}_{e_\alpha} \bbias^l(y_\alpha) \eigenvec{l}_{i}
+   \sum_{\alpha=1,2}   c'  \zeta^2 \omega_{y_\alpha}^2 \\
&=& \zeta \expec\left[S_{x}  \,|\,\ss_x = i\right]\
+ \zeta^2 \left (c + c' \right) \sum_{\alpha=1,2}   \omega_{y_\alpha}^2
\end{eqnarray*}
Take $c$ large enough so that $c + c'  < c(1+\eps')$ for some small $\eps' > 0$.
Moreover, from (\ref{eq:solutionRobust})
\begin{eqnarray*}
\omega_{y_1}^2 + \omega_{y_2}^2
&=& \left(\frac{\theta_{e_1}^2}{(\theta_{e_1}^2 + \theta_{e_2}^2)^2} + \frac{\theta_{e_2}^2}{(\theta_{e_1}^2 + \theta_{e_2}^2)^2}\right)(1 + O(\eps + \delta+ \eigenratio))\\
&=& \left(\frac{1}{\theta_{e_1}^2 + \theta_{e_2}^2}\right)(1 + O(\eps + \delta+\eigenratio))\\
&\leq& \frac{1}{2(\theta^*)^2}(1 + O(\eps + \delta+ \eigenratio)) < 1,
\end{eqnarray*}
where $\theta^* = e^{-g}$ so that $2(\theta^*)^2 > 1$.
Hence,
\begin{equation*}
\Gamma_{x}^i(\zeta)
\leq \zeta \expec[S_{x}\,|\,\ss_{x} = i] + c \zeta^2.
\end{equation*}
\end{proof}

\section{Concluding remarks}

We have shown how to reconstruct
latent tree Gaussian and GTR models
using $O(\log^2 n)$ samples in the
KS regime. In contrast, a straightforward
application of previous techniques
$O(\log^3 n)$ samples. Several
questions arise from our work:
\begin{itemize}
\item Can this reconstruction be done
using only $O(\log n)$ samples?
Indeed this is the case for the CFN model~\cite{Mossel:04a}
and it is natural to conjecture that it may
be true more generally.
However our current techniques
are limited by our need to use fresh samples
on each level of the tree to avoid unwanted
correlations between coefficients and samples
in the recursive conditions.

\item Do our techniques extend to general trees?
The boosted algorithm used here
has been generalized to non-homogeneous trees using a combinatorial algorithm
of~\cite{DaMoRo:11a}
(where edge weights are discretized to
avoid the robustness issues considered
in this paper). However general trees
have, in the worst case, linear diameters. To
apply our results,
one would need to control the depth of the subtrees
used for root-state estimation in the combinatorial
algorithm.
We leave this extension for future work.
\end{itemize}

\clearpage

\bibliographystyle{alpha}
\bibliography{thesis}

\end{document}